\documentclass[]{amsart}

\usepackage{amssymb}
\usepackage{amsmath}
\usepackage{amstext}
\usepackage{amsthm}
\usepackage{enumitem}

\usepackage{graphicx}
\usepackage{tabularx}

\usepackage[colorlinks=true]{hyperref} %

\usepackage{subcaption}

\usepackage[citestyle=alphabetic,bibstyle=alphabetic,backend=bibtex,maxbibnames=99]{biblatex}
\usepackage{tikz}
\usepackage{tikz-cd}
\usetikzlibrary{decorations.pathreplacing}
\usetikzlibrary{knots}

\usepackage[linesnumbered,ruled,vlined]{algorithm2e}

\SetCommentSty{mycommfont}

\newcommand{\tr}{\mathop{\mathrm{tr}}}
\newcommand{\rank}{\mathop{\mathrm{rank}}}
\newcommand{\spec}{\mathop{\mathrm{Spec}}}

\newcommand{\nil}{\mathop{\mathrm{nil}}}
\newcommand{\ideal}[1]{\langle#1\rangle}
\newcommand{\bs}[1]{\boldsymbol{#1}}

\newcommand{\etale}{{\'{e}tale }}

\newtheorem{theorem}{Theorem}[section]

\newtheorem{lemma}[theorem]{Lemma}
\newtheorem{proposition}[theorem]{Proposition}
\newtheorem{corollary}[theorem]{Corollary}
\theoremstyle{definition}
\newtheorem{definition}[theorem]{Definition}
\newtheorem{example}{Example}

\newtheorem{problem}{Problem}
\newtheorem*{problem*}{Problem}
\theoremstyle{remark}
\newtheorem{remark}{\indent Remark}
\newtheorem*{remark*}{Remark}

\title{An Effective Criterion for Covering Maps Between Real Varieties}

\author{Rizeng Chen}
\address{School of Mathematical Sciences, Peking University, 100871, Beijing, China}
\email{xiaxueqaq@stu.pku.edu.cn}

\keywords{Covering Map, Locally Constant Geometric Fiber, Finite \'Etale Morphism, Real Algebraic Geometry}

\subjclass[2020]{Primary 14Q30; Secondary 14P10, 14Q20, 68W30}
\thanks{This work is supported by National Key R\&D Program of China (No. 2022YFA1005102, No. 2024YFA1014003).}

\begin{document}
	
	\begin{abstract}

		We prove that a quasi-finite flat morphism with locally constant geometric-fiber cardinality between varieties over a real closed field induces a covering map on the rational points. The algebraic core of the criterion is a characteristic-zero result showing that every such morphism becomes finite étale after reduction. Since reduction does not change rational points, the induced map is a covering in the Euclidean topology. We extend the covering conclusion to Boolean combinations of closed subschemes and their complements, and construct a finite stratification separating finite covering families from families with positive-dimensional fibers. Finally, we compute the canonical non-finite, non-finite-flat, and non-finite-étale loci, using Gr\"{o}bner bases. This yields effective tests for the hypotheses of the covering criterion. We conclude with applications illustrating these constructions.
	\end{abstract}
	
	\maketitle

	\section{Introduction}
	
	\subsection{Background} Studying a family of polynomial systems is a central topic in applied algebraic geometry. For example, algebraic statistics studies the moduli space of all possible parameters of statistic models \cite{drton2009lectures,huh2014likelihood,rodriguez2015data}; in the theory of chemical reaction network, the steady states form a family parameterized by reaction rates \cite{craciun2009toric,muller2016sign,gross2021steady}; the dynamic behaviors of biological systems can be encoded as parametric polynomial systems \cite{wang2005stability}; and even the motion planning of robotics falls in this theme \cite{schwartz1983piano,canny1987complexity}. %

	Among all families, those whose projection to the parameter space is a covering map are perhaps the easiest to understand. In such a family, the cardinality of the fiber is locally constant, while the path-lifting property allows solutions to be transported continuously along paths in the parameter space. It is well-known that, under mild hypotheses, a family becomes a covering after restricting to a dense open subset of the parameter space. The exceptional systems are therefore parametrized by the complementary proper closed subset.
	
	This generic result, however, says little about the exceptional systems themselves. An illustrative example comes from algebraic computer vision: Sun, Li and Li showed that singular solutions of the Perspective-Three-Point (P3P) problem are associated with special configurations determined by the control-point triangle \cite{sun2026geometric}. Such a geometric description does not follow directly from the generic covering result. Moreover, the exceptional locus may itself fail to satisfy the hypotheses used to obtain that result, so the same argument cannot simply be applied again. It is therefore equally important to identify and understand the obstructions occurring in these exceptional systems.
	
	Hence, our goals in this paper are to give an explicit criterion for when a family induces a topological covering and to identify the obstructions that prevent an arbitrary family from satisfying this criterion. By ``explicit'', we also mean effectively computable, so the conditions can be checked by a computer.

	Now we can formally state our problem. In algebraic geometry, a family of polynomial systems is given by a morphism whose fibers correspond to each member in the family. %
	\begin{problem}\label{prob-main}
		Let $R$ be a real closed field. The map $\pi :X\to Y$ is an $R$-morphism of $R$-varieties. Recall that $X(R)$ and $Y(R)$ can be endowed with the Euclidean topology from the order structure on $R$. Find explicit effective conditions on $\pi$ so that the induced map on $R$-rational points $\pi_R: X(R)\to Y(R)$ is a covering map in the Euclidean topology.
	\end{problem}
	
	\subsection{Previous Results}
	
	There have been studies about covering maps between real varieties from the theoretical side. Delfs and Knebusch discussed covering maps in the category of locally semi-algebraic spaces \cite[Section 5]{delfs1984introduction}. Schwartz characterized a covering map $f$ of locally semi-algebraic spaces by the associated morphism of the corresponding real closed spaces being finite and flat \cite{schwartz1988open}. The theory of real closed spaces is introduced by Schwartz as a generalization of locally semi-algebraic spaces \cite{schwartz1989basic}. Baro, Fernando and Gamboa studied the relationship between branched coverings of semi-algebraic sets and its induced map on the spectrum of the ring of continuous semi-algebraic functions in \cite{baro2022spectral}. This line of research does not really give effective criteria because it relies on (various generalizations of) semi-algebraic spaces. The structural sheaf on a semi-algebraic space consists of continuous semi-algebraic functions and it is far from being computable (it is not even noetherian).
	
	There are also some effective criteria in the literature. Cylindrical algebraic decomposition (CAD) can be thought as coverings \cite{collins1975quantifier,hong1990improvement,mccallum1998improved,brown2001improved}. Indeed, the sections of a c.a.d.\ form a covering for the corank 1 projection. Using the implicit function theorem, Lazard and Rouillier showed that $f:X\to Y$ is a covering of $\mathbb{R}$-varieties in \cite{lazard2007solving} under the assumption that both $X$ and $Y$ are smooth, $X$ is equi-dimensional, $f$ is finite and all fibers are reduced. Recently, Helmer, Leykin and Nanda proposed an algorithm to stratify a quasi-finite morphism $f:X\to Y$ into a union of locally trivial fiber bundle in \cite{helmer2025effective}. Their algorithm computes Whitney stratification of both $X$ and $Y$ and a ``proper'' stratification of the map, so the conclusion follows from Thom's first isotopy lemma \cite[Proposition 11.1]{mather2012notes}.
	
	\subsection{Our Contribution}
	
	In this paper, we propose a new criterion for covering maps between real algebraic varieties. 
	\begingroup
	\def\thetheorem{\ref{thm-covering-map}}
	\begin{theorem}
		Let $\pi : X\to Y$ be a quasi-finite flat morphism with locally constant geometric fibers between $R$-varieties over a real closed field $R$. Then $\pi_R:X(R) \to Y(R)$ is a covering map in the Euclidean topology.
	\end{theorem}
	\endgroup
	The geometric picture of the criterion is clear. In algebraic geometry, the flatness serves as a notion of continuity in fibers. The local constancy of geometric fibers can be interpreted as the invariance of the number of distinct complex roots. Therefore these two conditions together shall imply that the fibers (solutions) are locally continuous functions on the base (in the parameters). 
	
	\begin{example} \label{examples-of-covering-maps}
		Here are some examples satisfying the assumption of Theorem \ref{thm-covering-map}. We shall see that they are indeed covering maps.           
		\begin{enumerate}[label=(\alph*)]
			\item %
			Set $Y$ to be the cuspidal curve with singularity removed $V(4 p^3+27 q^2)\backslash V(p,q)$ in the affine plane and set $X$ to be $V(4 p^3+27 q^2, x^3+px+q)\backslash V(p,q,x)$ in the three dimensional space. Then the projection from $X$ to $Y$ is quasi-finite flat with constant geometric fibers. Therefore the projection is a covering on the real points. See Figure \ref{fig:double-cover-of-cusp}. \label{example:double-cover-of-cusp-smooth}
			
			\item Consider $\mathbb{R}$-varieties $X=\spec \mathbb{R}[x,y,z,w]/\ideal{x^2+y^2-1,z^2+w^2-1,x-2z^2+1,x+2 w^2-1,y-2 z w}$ and $Y=\spec \mathbb{R}[x,y]/\ideal{x^2+y^2-1}$. The canonical projection $\pi:X\to Y$ is flat with constant geometric fiber size $2$. The $\mathbb{R}$-rational points of $Y$ can be modeled by the unit circle $S^1$ and $X(\mathbb{R})$ can be identified with a curve parameterized by
			$\varphi\mapsto(\cos 2\varphi,\sin 2\varphi,\cos \varphi,\sin\varphi)$ in the torus $S^1\times S^1=\{(\cos \theta,\sin \theta,\cos \varphi,\sin \varphi)|\theta,\varphi\in\mathbb{R}\}$. So $\pi_{\mathbb{R}}$ is given by           
			$$(\cos 2\varphi,\sin 2\varphi,\cos \varphi,\sin\varphi) \mapsto (\cos 2\varphi,\sin 2\varphi).$$ %
			Therefore $X(\mathbb{R})$ is a double cover of $S^1=Y(\mathbb{R})$. See Figure \ref{fig:no-global-section}.
			
			In this example, one can only define the semi-algebraic sections locally and $X(\mathbb{R})$ is not a disjoint union of copies of $Y(\mathbb{R})$. It is easy to see that the source is connected, but two different global sections will give two disjoint closed subsets. In the corank 1 case, sections can always be defined globally because of the order structure of $R$. In the general case, a covering map is the best we can get.
		\end{enumerate}
	\end{example}
	
	\begin{figure}[htbp]
		\centering              
		\begin{subfigure}[t]{0.30\linewidth}
			\includegraphics[width=\linewidth]{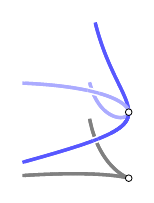}
			\caption{Double cover of cuspidal curve's smooth locus} 
			\label{fig:double-cover-of-cusp}
		\end{subfigure}
		~
		\begin{subfigure}[t]{0.36\linewidth}
			\includegraphics[width=\linewidth]{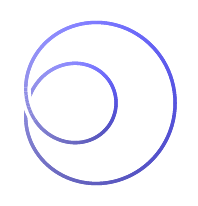}
			\caption{Non-trivial double cover of $S^1$} 
			\label{fig:no-global-section}
		\end{subfigure}
		~			
		\caption{Examples}
	\end{figure}
	
	The condition here is more general than Collins's cylindrical algebraic decomposition, which states for a corank $1$ projection from a hypersurface \cite[Theorem 1]{collins1975quantifier}, and ``discriminant variety'' in the sense of Lazard and Rouillier's that works with equidimensional smooth varieties \cite[Theorem 1]{lazard2007solving}. Therefore, these previous results can be viewed as special cases of our new criterion. Notably, this criterion is even applicable to singular varieties, thereby extending the scope of existing theory. See Example \ref{example:double-cover-of-nodal-curve} for a double cover of a singular curve.
	
	Actually, our new criterion is very sharp. Indeed if either assumption is dropped, the conclusion may fail.
	
	\begin{example}\label{example-non-covering}
		Theorem \ref{thm-covering-map} is quite sharp. We present some non-examples here.
		\begin{enumerate}[label=(\alph*)]
			\item The flatness assumption on the morphism is essential and natural. Let $X$ be the union of hyperbola $V(xy-1)$ and the origin $V(x,y)$ on the plane and consider its projection to the affine line. See Figure \ref{fig:non-flatness}. Obviously the projection is not a covering. Notice that the geometric fiber size is always 1. In this example, the projection is neither finite nor flat. 
			
			For a finite but not flat example, one can modify Example \ref{examples-of-covering-maps}\ref{example:double-cover-of-cusp-smooth} by setting $Y=V(4p^3+27 q^2)$ to be the whole cuspidal curve and adding a point above the singularity, e.g.\ $X=V(4p^3+27 q^2,x^3+px+q)\cup V(p,q,x-1)$. It might be instructive to recall Figure \ref{fig:double-cover-of-cusp}.
			
			\item Also, the constancy assumption on the geometric fiber cardinality is crucial. Let $X=V((y-1)^2-x)\cup V((y+1)^2+x)$ be the union of two parabolas and let $Y=\spec R[x]$ be the affine line, then the canonical projection from $X$ to $Y$ is obviously finite and flat. See Figure \ref{fig:real-fiber-insufficient}. Clearly this is not a covering either. Generically, the geometric fiber consists of four distinct complex roots $y=1\pm\sqrt{x},-1\pm\sqrt{x}$, but over $x=0$ there is only one pair of double roots. Notice that the number of real fibers over $Y(R)$ is always $2$ in this example. So even a constancy condition on the real fiber is not enough.
			
		\end{enumerate}         
	\end{example}
	
	\begin{figure}[htbp]
		\begin{subfigure}[t]{0.36\linewidth}
			\includegraphics[width=\linewidth]{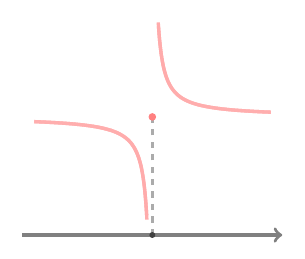}
			\caption{Without flatness} 
			\label{fig:non-flatness}
		\end{subfigure}
		~                       
		\begin{subfigure}[t]{0.36\linewidth}
			\includegraphics[width=\linewidth]{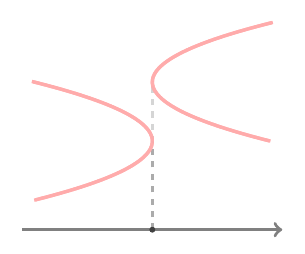}
			\caption{Geometric fiber not locally constant} 
			\label{fig:real-fiber-insufficient}
		\end{subfigure}
		~
		\caption{Non-examples}
	\end{figure}
	
	Our proof strategy consists of two steps. The first step might be of an independent interest: 
	\begingroup
	\def\thetheorem{\ref{thm-reduced-g-etale-is-etale}}
	\begin{theorem}
		Let $k$ be a field of characteristic $0$. If $\pi:X\to Y$ is a quasi-finite flat morphism with locally constant geometric fibers between $k$-varieties, then $\pi_{\mathrm{red}}:X_{\mathrm{red}}\to Y_{\mathrm{red}}$ is finite \'etale.
	\end{theorem}
	\endgroup
	
	If stronger smoothness conditions are imposed, then this result can be established very quickly. For example, suppose further that $Y$ is non-singular and $\pi$ is unramified, then a Shafarevich-style argument shows that locally $X$ is isomorphic to a closed subscheme of $\mathbb{A}_Y^1$ defined by a monic polynomial whose discriminant is a unit \cite[Section 2.6, Theorem 2.30]{shafarevich2013basic}.
	
	However, the lack of smoothness can cause great difficulty, as the normality of $Y$ plays an important role in Shafarevich's argument (essentially the property of minimal polynomial over an integrally closed ring, see \cite[Theorem 9.2]{matsumura1989commutative}). Notice that in general $X_{\mathrm{red}}\to Y_{\mathrm{red}}$ is not flat even though $X\to Y$ is, so Zariski's Main Theorem does not help either. %
	
	Our strategy is to use the trace pairing to separate the reduced structure of $X$ from its nilpotent structure. The local constancy of the geometric fibers forces this pairing to have constant rank, and its kernel turns out to be precisely the nilradical. It follows that the reduced morphism is finite flat with reduced fibers, and hence finite \'{e}tale.

	Therefore for a flat morphism of $k$-varieties ($\mathop{\mathrm{char}}k=0$), if the geometric fiber cardinality is finite and locally constant, then the ramification is purely due to the non-reduced structures of schemes and can be resolved by the reduction. %
	
	The second step finishes the proof by showing that a finite \'etale morphism over a real closed field does induce a covering map on the rational points. This is done by again retreating to the corank $1$ projection case and defining sections locally by solutions to a univariate equation.
	
	Besides the covering criterion itself, we develop two extensions that make it possible to study exceptional systems. First, Theorem \ref{thm-covering-map-of-subschemes} shows that the covering conclusion holds simultaneously for closed subschemes, their complements, and the Boolean combinations obtained from them, provided that the relevant restrictions are flat. We then drop the quasi-finiteness assumption in Proposition \ref{cor-triviality}: an arbitrary morphism admits a finite stratification of the base such that, on each stratum, either the induced maps on all these Boolean pieces are finite covering maps on the rational points, or every fiber of the ambient morphism is positive-dimensional. This separates the part accessible to the covering criterion from the obstruction caused by positive-dimensional fibers.
	
	Furthermore, we make the obstruction loci effectively computable in the affine case. Given a block-order Gr\"obner basis for the graph of a morphism, Theorem \ref{thm:non-finite-locus}, \ref{thm:non-finite-flat-locus}, and \ref{thm:non-finite-etale-locus} compute respectively the non-finite, non-finite-flat, and non-finite-\'{e}tale loci. These are intrinsic closed subsets of the parameter space, since their complements are the largest open subsets over which the morphism is respectively finite, finite flat, and finite \'{e}tale. The corresponding tests for finiteness, finite flatness, and finite \'{e}taleness follow as immediate corollaries. Finally, we illustrate these constructions through a likelihood correspondence from algebraic statistics and a symmetric matrix-completion problem.
	
	\subsection{Structure of the Paper}
	
	Section \ref{sect:preliminary} serves as a quick preparation of the geometric background needed for the paper. Section \ref{sect:reduced-map-finite-etale} is devoted to proving the technical result Theorem \ref{thm-reduced-g-etale-is-etale}. In Section \ref{sect:semi-algebraic-covering}, we prove the main covering criterion and develop the covering stratification. Then we develop algorithms to compute the non-finite, non-finite-flat, and non-finite-\'etale loci in Section \ref{sect:effective}. Some applications are discussed in Section \ref{sect-app}. Finally, we conclude the paper with some directions for future research in Section \ref{sect:future}.
	
	\section{Geometric Preliminaries} \label{sect:preliminary}
	
	In this section we assemble some geometric results that will be useful in this paper.
	
	\subsection{Conventions}~
	
	By a \textbf{$k$-variety}, we mean a separated, of finite type scheme over a field $k$.
	
	By a \textbf{corank 1 projection}, we mean a scheme morphism $X\to Y$ that factors through $\mathbb{A}_Y^1$ the affine line over $Y$ as a composition of closed immersion $X\hookrightarrow \mathbb{A}_Y^1$ and the canonical projection $\mathbb{A}_Y^1\to Y$.
	
	We fix $R$ to be a real closed field and $C=R[\sqrt{-1}]$ is the algebraic closure of $R$. 
	
	By a \textbf{covering map}, we mean a continuous map $p:E\to B$ such that every $b\in B$ has an open neighborhood $U$ whose preimage is a disjoint (possibly empty) union of opens homeomorphic to $U$. We do \emph{not} assume $p$ to be surjective here because $E$ and $B$ are usually taken as the rational points of varieties and they sometimes can be empty.
	
	\subsection{Geometric Points and Geometric Fibers}~
	
	A \textbf{geometric point} $\overline{x}$ of a scheme $X$ is a morphism $\overline{x}: \spec \Omega\to X$, where $\Omega$ is an algebraically closed field. We say $\overline{x}$ lies over $x$ to indicate that $x\in X$ is the image of $\overline{x}$. 
	
	Given a morphism of schemes $f:X\to Y$ and a geometric point $\overline{y}:\spec \Omega\to Y$, the \textbf{geometric fiber} $X_{\overline{y}}$ is the fiber product $X\times_Y \spec \Omega$.
	
	For a quasi-finite scheme morphism $f:X\to Y$, the function $n(y): \begin{array}{rcl} Y& \to &\mathbb{N} \\  y&\mapsto & \# X_{\overline{y}} \end{array}$ counts the cardinality of the geometric fiber. If this function is locally constant on $Y$, then we say $f$ is a morphism with locally constant geometric fibers.
	
	The following result about geometric fiber due to Grothendieck is very useful.
	\begin{proposition}[Grothendieck, {\cite[Proposition 15.5.9]{grothendieck1966elements}}]\label{prop-grothendieck}
		Let $f:X\to Y$ be a flat morphism locally of finite presentation. For each $y\in Y$, let $n(y)$ be the number of geometric connected components of $f^{-1}(y)$.
		\begin{enumerate}[label*=(\roman*)]
			\item If $f$ is quasi-finite and separated, then the function $y\mapsto n(y)$, which is the number of geometric points in $f^{-1}(y)$, is lower semi-continuous on $Y$. If it is a constant in a neighborhood of $y_0$, then $f$ is proper (hence finite) on a neighborhood of $y_0$.
			\item If $f$ is proper, then function $y\mapsto n(y)$ is lower semi-continuous. Suppose moreover $f^{-1}(y_0)$ is geometrically reduced over $\kappa_{y_0}$, then $n(y)$ is a constant on a neighborhood of $y_0$.
		\end{enumerate}
	\end{proposition}
	
	\subsection{Flat and \'Etale Morphisms}
	
	In algebraic geometry, a morphism of schemes parameterizes a family of schemes and the fibers are family members. The challenge is that those fibers may ``jump'' at some point in a weird way. The \textbf{flatness} prevents this so a flat family is continuously varying.
	
	\begin{definition}
		Let $\pi :X\to Y$ be a morphism of schemes. We say $\pi$ is flat over $Y$ at a point $x\in X$, if the stalk $\mathcal{O}_{X,x}$ is a flat $\mathcal{O}_{Y,y}$-module, where $y=\pi(x)$ is the image of $x$. That is, for any short exact sequence of $\mathcal{O}_{Y,y}$-modules
		$$0\to M\to N\to P \to 0,$$
		the tensor product preserves exactness:
		$$0\to M \otimes_{\mathcal{O}_{Y,y}} \mathcal{O}_{X,x}\to N \otimes_{\mathcal{O}_{Y,y}} \mathcal{O}_{X,x}\to P \otimes_{\mathcal{O}_{Y,y}} \mathcal{O}_{X,x} \to 0.$$
		
		We say $\pi$ is flat, if it is flat at every point of $X$.
	\end{definition}
	
	The notion of ``\textbf{\'etaleness}'' is an algebraic analogue of covering spaces. The French word ``\'etale'' refers to the appearance of the calm sea surface at high tide under a full moon \cite[p.\ 175]{mumford1999red}, where the moonlight reflects off the locally parallel waves. 
	
	\begin{definition}
		Let $\pi:X\to Y$ be a morphism of $k$-varieties over a field $k$. The morphism $\pi$ is said to be \'etale if one of the following equivalent conditions holds:
		\begin{enumerate}
			\item $\pi$ is flat, and all geometric fibers of $\pi$ are regular of dimension $0$.
			\item $\pi$ is flat and unramified (for all $x\in X$ and $y=\pi(x)$: $\mathfrak{m}_y\mathcal{O}_{X,x}=\mathfrak{m}_x$, and $\kappa_x$ is a separable algebraic extension of $\kappa_y$).
			\item $\pi$ is flat, and the relative differential $\Omega_{X/Y}=0$.
			\item $\pi$ is smooth of relative dimension $0$.
			\item ($\mathop{\mathrm{char}}k=0$). $\pi$ is flat, and all fibers of $\pi$ are regular of dimension $0$.
			\item ($\mathop{\mathrm{char}}k=0$). $\pi$ is flat, and all fibers of $\pi$ are reduced of dimension $0$.
		\end{enumerate}
	\end{definition}

	\section{Flat Morphism with Locally Constant Finite Geometric Fibers} \label{sect:reduced-map-finite-etale}
	In this section, we will show the following: a quasi-finite flat morphism with locally constant geometric fibers between $k$-varieties ($\mathop{\mathrm{char}}k=0$) becomes finite \'etale after reduction. We shall establish some basic properties first.	
	
	\begin{proposition}\label{prop-qff-implies-ff} %
		A quasi-finite flat $k$-morphism with locally constant geometric fibers between $k$-varieties is finite.        
	\end{proposition}
	\begin{proof}
		Denote the morphism by $f:X\to Y$. Since $n(y)$ is locally constant, for any $y\in Y$ there is a neighborhood $U$ of $y$ such that $f^{-1}(U)\to U$ is finite by Proposition \ref{prop-grothendieck}.%
	\end{proof}
	
	\begin{proposition}\label{prop:finite-etale-has-constant-fibers}
		A finite \'etale $k$-morphism $f:X\to Y$ of $k$-varieties is quasi-finite, flat and has locally constant geometric fibers.
	\end{proposition}
	\begin{proof}
		The map $f$ is clearly proper and flat. By the \'etaleness, the geometric fibers are reduced. Therefore by Proposition \ref{prop-grothendieck}, the function $n(y)$ is locally constant.
	\end{proof}
	
	Now we are ready for the main result in this section.
	\begin{theorem}\label{thm-reduced-g-etale-is-etale}
		Let $k$ be a field of characteristic 0. Suppose $X,Y$ are $k$-varieties and $\varphi:X\to Y$ is a quasi-finite flat $k$-morphism with locally constant geometric fibers, then the reduced map $\varphi_{\mathrm{red}}: X_{\mathrm{red}} \to Y_{\mathrm{red}}$ is finite \'etale.
		
	\end{theorem}	%
	
	\begin{proof}We divide the proof into a few steps.
		
		\textit{Step 1. Reduce to the affine case.}
		First notice that $\varphi$ is actually already finite by Proposition \ref{prop-qff-implies-ff} (so is $\varphi_{\mathrm{red}}$).    
		The question is local, so we may assume $X=\spec B$, $Y=\spec A$ are affine varieties and $B=A[x_1,\ldots,x_n]/I$ is a free $A$-module of finite rank $r>0$ (recall that for a finite module over a noetherian ring, being flat, being projective and being locally free are all equivalent \cite[\href{https://stacks.math.columbia.edu/tag/00NX}{Lemma 00NX}]{stacks-project}). Since $\# X_{\overline{y}}$ is locally constant, we may further assume that the geometric fiber cardinality is a constant $s$. Replacing $A$, $B$ with $A_{\mathrm{red}}=A/\nil A$ and $B\otimes_A A_{\mathrm{red}}$ respectively, we may even assume that $A$ is reduced. 
		\medskip
		
	\textit{Step 2. The trace pairing.}
	Recall that for any given $b\in B$, the multiplication map $L_b:B\to B, b'\mapsto bb'$ is an $A$-module endomorphism. Because $B$ is finite free over $A$, there is a trace map $\tr_{B/A}: B\to A$, sending $b$ to the trace of $L_{b}$. Consider the trace-pairing map
	$$
	 \begin{array}{rcl}
	\Theta:	B&\to& B^\vee:=\operatorname{Hom}_A(B,A) \\
		b&\mapsto&\bigl(b'\mapsto	\tr_{B/A}(bb')\bigr).
	\end{array}
	$$

	It is a classical result that the rank of $\Theta$ at $p\in Y=\spec A$ is exactly the number of distinct points in the geometric fiber $X_{\overline{p}}$ (see e.g.\ \cite[Theorem 4.102]{basu2008algorithms}).
	Thus $\Theta$ has constant fiber rank $s$.
	\medskip
	
	\textit{Step 3. The kernel of $\Theta$ is locally free.}
	Set $M=\operatorname{coker}\Theta.$
	The module $M$ has the explicit finite presentation
	$$
	B\xrightarrow{\Theta}B^\vee\longrightarrow M\longrightarrow 0.
	$$
	Choose bases of $B$ and $B^\vee$, and let $T$ be the resulting
	$r\times r$ matrix. Since $T$ has rank $s$ over every residue field, every
	$(s+1)\times(s+1)$ minor of $T$ lies in every prime ideal of $A$ but at every prime there is some non-zero $s\times s$ minor. Therefore, the Fitting ideals of $M$ are
	$$
	F_{r-s}(M)=A,\ 
	F_{r-s-1}(M)=0.
	$$
	It follows that $M$ is locally free of rank $r-s$ \cite[\href{https://stacks.math.columbia.edu/tag/07ZD}{Lemma 07ZD}]{stacks-project}.	
	
	Since $M$ is finite locally free, it is projective. Therefore the exact
	sequence
	$$
	0\longrightarrow \mathop{\mathrm{im}}\Theta
	\longrightarrow B^\vee
	\longrightarrow M
	\longrightarrow 0
	$$
	splits by the property of projective modules (see e.g.\ \cite[Theorem A3.1]{eisenbud2013commutative} or \cite[\href{https://stacks.math.columbia.edu/tag/013C}{Lemma 013C}]{stacks-project}). Thus $\mathop{\mathrm{im}}\Theta$ is a direct summand of the finite free module
	$B^\vee$, and it is finite projective, equivalently finite locally free. The rank of $\mathop{\mathrm{im}}\Theta$ is $\rank B^\vee -\rank M=s$.
	
	Set	$K=\ker\Theta$. Then the exact sequence
	$$
	0\longrightarrow K
	\longrightarrow B
	\longrightarrow\mathop{\mathrm{im}}\Theta
	\longrightarrow 0
	$$
	also splits for the same reason. Consequently $K$ is finite locally free of rank $r-s$.
	\medskip
	
	\textit{Step 4. The kernel of $\Theta$ is the nilradical of $B$.}
	We next show that
	$K=\nil B$.
	
	First, $K$ is an ideal of $B$: if $b\in K$ and $b'\in B$, then the functional $\Theta(bb')$ is zero: for every
	$b''\in B$, we have
	$$
		\Theta(bb')(b'')=\tr\nolimits_{B/A}((bb')b'')=\tr\nolimits_{B/A}(b(b'b''))=\Theta(b)(b'b'')=0.
	$$
	
	On the one hand, $\nil B\subseteq K$. If $b\in B$ is nilpotent, then $bb'$ is also nilpotent for every $b'\in B$, and the multiplication by $bb'$ is a nilpotent endomorphism. This implies that the trace of $L_{bb'}$ is zero (this can be checked by passing to the fibers and using reducedness of $A$). Therefore $b\in K$.
	
	On the other hand, $\nil B\supseteq K$. Let $q\in \spec B$ and $p=q\cap A \in \spec A$. We will show that every element $b\in K$ evaluates to $0$ in $\kappa_q$. Since $b$ is in the radical of the trace pairing $\Theta$ and trace pairing commutes with base change, we can see that 
	$$b(p)\in B(p):=B\otimes_A \kappa_p$$ is in the radical of the trace pairing $\Theta(p): B(p) \to \mathop{\mathrm{Hom}}_{\kappa_{p}}(B(p),\kappa_p)$. We claim that $b(p)\in \nil B(p)$. Indeed, for every $i\geq 1$: $$b^i\in K,\quad  \tr\nolimits_{B(p)/\kappa_p}\left(b^i(p)\right)=0.$$ By Newton's identity, the characteristic polynomial of $b(p)$ is then $\lambda^r=0$. Therefore Cayley-Hamilton theorem implies that $(b(p))^r=0$ and $b(p)\in \nil B(p)$. Since $b(p)$ is nilpotent and evaluating at $q$ factors through $B(p)$, we can conclude that $b(q)=0\in \kappa_q$. By the formal Nullstellensatz, $b\in \nil B$.
	
	\textit{Step 5. The \'etaleness.}
	Now we know that $K=\nil B$. So 
	$$B_{\mathrm{red}}=B/\nil B=B/K\cong \mathop{\mathrm{im}}\Theta$$ is a finite locally free module of rank $s$. The flatness is established. Since the rank of $B_{\mathrm{red}}$ as an $A$-module is $s$, for any $p\in\spec A$, the degree of the fiber $(X_{\mathrm{red}})_p$ is $s$, which is also the number of distinct geometric points. Therefore the fiber must be reduced, hence regular of dimension $0$. By definition, $X_{\mathrm{red}}$ is \'etale over $Y$.

	\end{proof}

	\begin{example} \label{example:double-cover-of-cusp}
		Suppose $k$ is a field of characteristic 0. Consider $$X=\spec k[p,q,x]/\ideal{4p^3+27q^2,x^3+px+q}, Y=\spec k[p,q]/\ideal{4p^3+27q^2}$$ and the canonical projection $\pi:X\to Y$. Clearly $\pi$ is finite and flat. See Figure \ref{fig:depressed-cubic-over-discriminant}. %
		
		\begin{figure}[hbtp]
			\centering
			\includegraphics[width=0.6\linewidth]{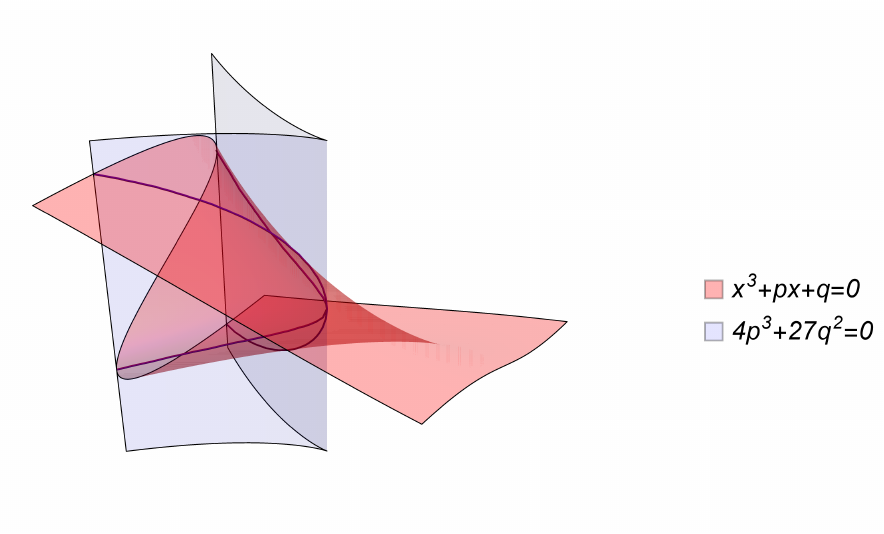}
			\caption{$X$ as the intersection of zeros of defining equations} 
			\label{fig:depressed-cubic-over-discriminant}
		\end{figure}
		
		Let $U=Y_{\mathrm{reg}}=Y\backslash V(p,q)$, then $\pi|_{\pi^{-1}(U)}: X\times_Y U\to U$ is finite, flat and $\# X_{\overline{y}}=2$ for all $y\in U$. So by Theorem \ref{thm-reduced-g-etale-is-etale}, the reduced map $(X\times_Y U)_{\mathrm{red}}\to U_{\mathrm{red}}$ is finite \'etale.
		
		Notice that in general, even the flatness is not preserved under reduction. In this example, $\pi_{\mathrm{red}}: X_{\mathrm{red}}\to Y_{\mathrm{red}}$ is not flat! In fact, by a direct computation, one can show that $X_{\mathrm{red}}=\spec k[p,q,x]/\ideal{6 p x^2-9 q x+4 p^2,x^3+ p x+q}$ and all fibers of $\pi_{\mathrm{red}}$ have degree $2$, except the fiber over the origin, which has degree $3$. %
	\end{example}

	\begin{example}
		When $\mathop{\mathrm{char}} k=p>0$, the same conclusion does not hold anymore. Let $X=\spec \mathbb{F}_2[t][x]/\ideal{x^2+t}$ and $Y=\spec \mathbb{F}_2[t]$, then $X$ and $Y$ are both reduced, the canonical projection $\pi:X\to Y$ is flat and the cardinality of geometric fiber is always $1$. But $\pi$ is not smooth.
	\end{example}   
	
	The following corollary is immediate from Proposition \ref{prop:finite-etale-has-constant-fibers} and Theorem \ref{thm-reduced-g-etale-is-etale}.
	
	\begin{corollary} \label{cor:equiv-characterization-of-finite-etale}
		Let $k$ be a field of characteristic $0$. Suppose $X,Y$ are $k$-varieties and the $k$-morphism $\varphi: X\to Y$ is quasi-finite and flat. Then the following are equivalent:
		\begin{enumerate}
			\item $\varphi_{\mathrm{red}}:X_{\mathrm{red}}\to Y_{\mathrm{red}}$ is finite \'etale.
			\item The geometric fiber cardinality $n(y)$ is a locally constant function on $Y$.
		\end{enumerate}
	\end{corollary}
	
	We end this section with a corollary that concerns the geometric fiber cardinality of a subvariety.
	
	\begin{corollary}\label{cor:subvariety-covering}
		Let $k$ be a field of characteristic 0. Suppose $X,Y$ are $k$-varieties and $\varphi:X\to Y$ is a quasi-finite flat $k$-morphism with locally constant geometric fibers. Let $i:Z\hookrightarrow X$ be a closed subscheme of $X$ and $j:U=X\backslash Z\hookrightarrow X$ be the complement with the open subscheme structure.
		
		Suppose that $\psi=\varphi\circ i: Z\hookrightarrow X \to Y$ is flat. Then 
		\begin{enumerate}
			\item The composition $\psi$ is also a finite flat morphism with locally constant geometric fibers, and $\psi_{\mathrm{red}}:Z_{\mathrm{red}}\to Y_{\mathrm{red}}$ is finite \'etale.        	
			\item Similarly, $\varphi \circ j$ is finite flat with locally constant geometric fibers, and $\varphi_{\mathrm{red}} \circ j_{\mathrm{red}}: U_{\mathrm{red}}\to Y_{\mathrm{red}}$ is finite \'etale.        	
			\item The closed immersion $i_{\mathrm{red}}: Z_{\mathrm{red}}\hookrightarrow X_{\mathrm{red}}$ and the open immersion $j_{\mathrm{red}}:U_{\mathrm{red}}\hookrightarrow X_{\mathrm{red}}$ are also finite \'etale.
		\end{enumerate}                 
	\end{corollary}
	\begin{proof}~
		
		\begin{enumerate}
			\item We can compute $\#Z_{\overline{y}}$ in two different ways. Notice that $\varphi:X\to Y$ is actually finite by Proposition \ref{prop-qff-implies-ff}.
			
			On the one hand, $\psi:Z\to Y$ is finite and flat, so $\#Z_{\overline{y}}$ is lower semi-continuous by Proposition \ref{prop-grothendieck}.%
			
			On the other hand, $\# Z_{\overline{y}}=\#(Z\times_X X_{\mathrm{red}})_{\overline{y}}$. By Theorem \ref{thm-reduced-g-etale-is-etale}, $\varphi_{\mathrm{red}}:X_{\mathrm{red}}\to Y_{\mathrm{red}}$ is finite \'etale. So the fiber $(X_{\mathrm{red}})_y$ is zero-dimensional and regular, thus reduced. This implies that the subscheme $(Z\times_X X_{\mathrm{red}})_y$ is also zero-dimensional and reduced. So $\#(Z\times_X X_{\mathrm{red}})_{\overline{y}}=\deg (Z\times_X X_{\mathrm{red}})_y$ and we have
			$$\deg (Z\times_X X_{\mathrm{red}})_y=\dim_{\kappa_y} \left(\varphi_*\left(i_*\mathcal{O}_Z \otimes_{ \mathcal{O}_X} {(\mathcal{O}_{X})}_{\mathrm{red}}\right)_y \otimes_{\mathcal{O}_{Y,y}} \kappa_y\right)$$  
			is upper semi-continuous \cite[Exercise II.5.8]{hartshorne2013algebraic}.
			
			Therefore $\#Z_{\overline{y}}$ is an integer-valued continuous function, thus a locally constant function. By Theorem \ref{thm-reduced-g-etale-is-etale} again, $\psi_{\mathrm{red}}$ is finite \'etale.
			\item Observe that $\# U_{\overline{y}}=\# X_{\overline{y}}-\# Z_{\overline{y}}$. Therefore $\# U_{\overline{y}}$ is also locally constant. Notice that the open immersion $j$ is quasi-finite and flat. So $\varphi \circ j$ is a quasi-finite flat morphism with locally constant geometric fibers, hence finite by Proposition \ref{prop-qff-implies-ff}. Once more, Theorem \ref{thm-reduced-g-etale-is-etale} implies that $\varphi_{\mathrm{red}} \circ j_{\mathrm{red}}: U_{\mathrm{red}}\to Y_{\mathrm{red}}$ is finite \'etale.
			\item  Combining \cite[I. Corollary 3.6]{milne1980etale} and \cite[Exercise II.4.8]{hartshorne2013algebraic} (or \cite[\href{https://stacks.math.columbia.edu/tag/02GW}{Lemma 02GW}, \href{https://stacks.math.columbia.edu/tag/035D}{Lemma 035D}]{stacks-project}), we see that there is a cancellation law of finite \'etale morphisms. Namely, if $g\circ f$ and $g$ are finite \'etale morphisms between $k$-varieties, then so is $f$. Therefore the closed immersion $i_{\mathrm{red}}: Z_{\mathrm{red}}\to X_{\mathrm{red}}$ is finite \'etale, as $\psi_{\mathrm{red}}= \varphi_{\mathrm{red}} \circ i_{\mathrm{red}}$ and $\varphi_{\mathrm{red}}$ are both finite \'etale. Similarly, $j_{\mathrm{red}}$ is also finite \'etale.
		\end{enumerate}                         
	\end{proof}		
	
	\section{Semi-algebraic Covering Maps} \label{sect:semi-algebraic-covering}
	In this section, we will consider the case where $k=R$ is a real closed field and we will show that a quasi-finite flat morphism of $R$-varieties with locally constant geometric fibers induces a covering map on the $R$-rational points in the Euclidean topology.		
	
	We first prove the covering criterion and compare it with previous constructions. We then discuss its complex	analogue and establish a version for subsets obtained from closed and open subschemes. Finally, we show that an arbitrary morphism admits a finite stratification into covering strata and strata with positive-dimensional fibers, and derive a sample-point consequence.
	
	\begin{theorem}\label{thm-covering-map}
		Let $X$, $Y$ be two $R$-varieties. Suppose that $\pi: X\to Y$ is a quasi-finite flat $R$-morphism with locally constant geometric fibers. Let $y\in Y(R)$. Then there exists $r\in \mathbb{N}$, semi-algebraically connected Euclidean neighborhood $U\subseteq Y(R)$ of $y$ and semi-algebraic continuous functions $\xi_1,\ldots,\xi_r:U\to X(R)$ such that
		\begin{enumerate}
			\item The real fiber over every $y\in U$, has cardinality $r$.
			\item $\xi_i$ are sections of $\pi$, that is $\pi\circ\xi_i=\mathrm{id}_U$ for $i=1,\ldots,r$.
			\item For every $y\in U$: $i\ne j$ implies $\xi_i(y)\ne \xi_j(y)$.
		\end{enumerate}
		Therefore $\pi_R$ is a covering map in the Euclidean topology.
	\end{theorem}
	\begin{proof}
		We may replace $\pi:X\to Y$ with $\pi_{\mathrm{red}}:X_{\mathrm{red}}\to Y_{\mathrm{red}}$ and assume that $\pi$ is a finite \etale morphism by Theorem \ref{thm-reduced-g-etale-is-etale}. By \cite[Proposition 18.4.5]{grothendieck1967elements}, locally $\pi$ is of the form $\spec A[\lambda]/\ideal{u}\to \spec A$, where $u\in A[\lambda]$ is a monic polynomial and $u'$ is invertible modulo $u$ (this is a variant of the fact that \'etale morphisms are locally standard \'etale). Now $\spec A[\lambda]/\ideal{u}\to \spec A$ is a finite free morphism of affine $R$-varieties with constant geometric fibers, so we can define $\xi_i$ to be the $i$-th largest real root of $u$ in a Euclidean neighborhood of $y$ and they are continuous semi-algebraic functions by \cite[Theorem 5.7]{chen2023geometric}.
	\end{proof}
	
	\begin{remark}
		For readers familiar with locally semi-algebraic spaces, it can be immediately concluded that $\pi_R$ is a covering map after applying Theorem \ref{thm-reduced-g-etale-is-etale}, see \cite[Example 5.5]{delfs1984introduction}.
	\end{remark}
	
	\begin{remark}
		One possible way to understand Theorem \ref{thm-covering-map} is to compare it with the Ehresmann's fibration theorem \cite{ehresmann1950connexions}. Actually,
		Theorem \ref{thm-covering-map} implies a semi-algebraic weaker variant of Ehresmann's fibration theorem. Suppose that $\pi:X\to Y$ is a proper morphism of non-singular $R$-varieties with surjective tangent maps $T_\pi: T_x X\to T_{\pi(x)} Y \otimes_{\kappa_y} \kappa_x$, and suppose further that all fibers of $\pi$ are finite, then $\pi_R$ is a covering map in the Euclidean topology. Indeed, such a morphism is  quasi-finite, proper and \'etale. Therefore locally there are semi-algebraic sections to $\pi_R$ by Theorem \ref{thm-covering-map}.
	\end{remark}		
	
	The assumption in Theorem \ref{thm-covering-map} is purely algebro-geometric and independent of the order structure of $R$, yet it is strong enough to conclude the local existence of semi-algebraic sections. It might be instructive to recall Example \ref{examples-of-covering-maps} and Example \ref{example-non-covering} in the introduction. Of course, such a simple condition is not always necessary, as the map $x\mapsto x^3$ is a semi-algebraic homeomorphism but the geometric fiber cardinality drops when $x=0$. Also, one may add embedded points to destroy the flatness without changing the map on the topological level. %
	
	Still, Theorem \ref{thm-covering-map} is quite strong among the results of the same kind in the sense that there is no assumption on $X$ or $Y$, so non-reduced, reducible, singular varieties with irreducible components of various dimensions are all allowed. Also, we can work over arbitrary real closed field instead of just real numbers $\mathbb{R}$. Most previous results producing covering maps are from a differential view, relying on smooth structures. Therefore $X,Y$ are required to be equidimensional smooth $\mathbb{R}$-varieties. %
	
	We first compare Theorem~\ref{thm-covering-map} with the discriminant variety of Lazard and Rouillier. Their construction removes explicit loci from the base so that the remaining morphism is finite, flat and unramified. Theorem~\ref{thm-covering-map} therefore recovers their covering result as follows.

	\begin{corollary}[Lazard-Rouillier {cf.\ \cite[Theorem 1]{lazard2007solving}}]\label{cor-disc-var}
		Let $X=\spec B,Y=\spec A$ be two $R$-varieties. Suppose $X$ is a distinguished open $D(f)$ of a closed subscheme $V(I)\subseteq \mathbb{A}_Y^n=\spec A[x_1,\ldots,x_n]$, and the projection $\pi:X\hookrightarrow \mathbb{A}_Y^n \to Y$ is dominant.
		
		Let $\overline{V(I)}\subseteq \mathbb{P}_Y^n$ be the projective closure of $V(I)$, and let $H=\mathbb{P}_Y^n \backslash \mathbb{A}_Y^n$ be the hyperplane at infinity.
		Define
		\begin{enumerate}[label=(\alph*)]
			\item $W_\infty \subseteq Y$ to be the scheme-theoretic image of $\overline{V(I)}\cap H$ in $Y$.
			\item $W_{\mathcal{F}}$ to be the scheme-theoretical image of $V(\ideal{f}+I)$,
			\item $W_{\mathrm{sing}}\subseteq Y$ to be the singular locus of $Y$,
			\item $W_{\mathrm{sd}}$ to be the scheme-theoretical image of irreducible components of $X$ of dimension $<\dim X$,              
			\item $W_{\mathrm{c}_1}$ to be the scheme-theoretical image of the singular locus of $X$,
			\item $W_{\mathrm{c}_2}$ to be the scheme-theoretical image of the ramification locus $\Delta(X/Y)=\mathop{\mathrm{Supp}}\Omega_{X/Y}$, the support of the sheaf of relative differentials on $X$.
		\end{enumerate}
		
		Then $\pi_R:X(R)\to Y(R)$ is a covering map outside the union $W=W_\infty \cup W_{\mathrm{sing}} \cup W_{\mathrm{sd}} \cup W_{\mathcal{F}} \cup W_{\mathrm{c}_1} \cup W_{\mathrm{c}_2}$. Moreover, if $R=\mathbb{R}$, then $\pi_R$ is an analytic covering outside $W$.
	\end{corollary}
	\begin{proof}
		First, we retreat to the finite case. Clearly, the morphism $\overline{V(I)}\hookrightarrow \mathbb{P}_Y^n \to Y$ is projective, hence proper \cite[Theorem II.4.9]{hartshorne2013algebraic} (or \cite[\href{https://stacks.math.columbia.edu/tag/01WC}{Lemma 01WC}]{stacks-project}). Then the base change to $Y\backslash W_\infty$ is also proper. Notice that $\overline{V(I)} \times_Y (Y\backslash W_\infty)=V(I) \times_Y (Y\backslash W_\infty)$. Therefore, the projection $V(I) \times_Y (Y\backslash W_\infty) \to Y \backslash W_\infty$ is affine and proper, hence finite \cite[Exercise II.4.6]{hartshorne2013algebraic} (or \cite[\href{https://stacks.math.columbia.edu/tag/01WN}{Lemma 01WN}]{stacks-project}). After another fiber product by the open subscheme $Y\backslash W_{\mathcal{F}}$, we may now assume that $\pi:X\to Y$ is a finite surjective morphism.
		
		Next, we are going to show that $\pi$ becomes flat after removing $W_{\mathrm{sing}}$, $W_{\mathrm{sd}}$ and $W_{\mathrm{c}_1}$. Recall the miracle flatness theorem \cite[Theorem 23.1]{matsumura1989commutative} (or \cite[\href{https://stacks.math.columbia.edu/tag/00R4}{Lemma 00R4}]{stacks-project}): for a morphism $\pi:X\to Y$ with a Cohen-Macaulay source $X$ and a regular base $Y$, $\pi$ is flat if and only if for all $x\in X$ and $y=\pi(x)\in Y$, 
		$$\dim \mathcal{O}_{X,x}=\dim \mathcal{O}_{Y,y}+\dim \mathcal{O}_{X_y,x}.$$ 
		Now after removing $W_{\mathrm{sing}}$, $W_{\mathrm{sd}}$ and $W_{\mathrm{c}_1}$, our source $X$ is regular (hence Cohen-Macaulay by \cite[Theorem 17.8]{matsumura1989commutative} or \cite[\href{https://stacks.math.columbia.edu/tag/00NQ}{Lemma 00NQ}]{stacks-project}) and of pure dimension $d$, our base $Y$ is regular. Therefore to show the flatness, it suffices to show the equation in dimensions holds everywhere. Notice that $Y$ is also of pure dimension $d$ (choose an irreducible component $Y_i$ of $Y$, there must be an irreducible component $X_i$ of $X$, mapped onto $Y_i$ as $\pi$ is finite surjective, so $\dim X=\dim X_i=\dim Y_i$). Then we have 
		$$\begin{array}{lclr}
			\dim \mathcal{O}_{X,x} &=& \dim X-\dim \overline{\{x\}} & X \text{ equi-dimensional} \\
			&=&\dim Y-\dim \overline{\{y\}} & \pi \text{ finite surjective}\\
			&=&\dim \mathcal{O}_{Y,y} & Y \text{ equi-dimensional} \\
			&=&\dim \mathcal{O}_{Y,y}+\dim \mathcal{O}_{X_y,x} & X_y \text{ finite}
		\end{array}$$
		
		Finally, $W_{\mathrm{c}_2}$ is removed so $\pi$ is finite, flat, unramified. Therefore $\pi|_{\pi^{-1}(Y\backslash W)}: X\times_Y(Y\backslash W)\to Y\backslash W$ is finite \'etale . Our conclusion follows from Theorem \ref{thm-covering-map}.
		
		When $R=\mathbb{R}$, we can talk about analytic properties of morphisms. Since $\pi|_{\pi^{-1}(Y\backslash W)}: X\times_Y(Y\backslash W)\to Y\backslash W$ is a finite \'etale morphism between smooth varieties, its tangent map is surjective. It follows immediately from inverse function theorem that $\pi_R$ is a local diffeomorphism.
	\end{proof}
	
	\begin{remark} \label{remark:minimal-non-finite-etale-locus}
		The corollary can be sharpened. After removing $W_\infty$ and $W_{\mathcal{F}}$, we may assume that $\pi:X\to Y$ is a finite morphism. In this case, the flat locus $F\subseteq Y$ of $\pi$ is an open subset given explicitly by the Fitting ideals \cite[Corollary 7.3.13]{greuel2008singular}. Recall that $n(y)$ is lower semi-continuous on $F$ by Proposition \ref{prop-grothendieck}. %
		Therefore there is a biggest open subset $F'\subseteq F$ such that the $n(y)$ is locally constant on $F'$. On $F'$, $\pi$ induces a covering map (but not necessarily an analytic cover).
	\end{remark}

	The complex analogue of Theorem \ref{thm-covering-map} is also true. It can be proved using Weil restriction (restriction of scalars) (see, e.g.\ \cite[Section 4]{scheiderer2006real} for an introduction). Roughly speaking, the Weil restriction $\mathop{\mathrm{Res}_{C/R}}$ of a $C$-variety is the $R$-variety obtained by separating the real and imaginary parts of the coordinates.
	
	\begin{theorem}
		Let $X$, $Y$ be two $C$-varieties. Suppose that $\pi: X\to Y$ is a quasi-finite flat $C$-morphism with locally constant geometric fibers. Then $\pi_C$ is a covering map in the Euclidean topology. %
	\end{theorem}
	\begin{proof}
		Replacing $\pi:X\to Y$ with the reduced map, we may assume that $\pi$ is finite and \'etale. Because the question is local, we may also assume that $X$, $Y$ are affine varieties. Since $C/R$ is a finite field extension, the Weil restriction $\mathop{\mathrm{Res}_{C/R}}$ exists for quasi-projective varieties \cite[Corollary 4.8.1]{scheiderer2006real}. Let $Z=\mathop{\mathrm{Res}_{C/R}}{X}$ and $W=\mathop{\mathrm{Res}_{C/R}}{Y}$ be the Weil restrictions of $X$ and $Y$ respectively, then $Z$ and $W$ are $R$-varieties whose $R$-rational points can be naturally identified with closed points of $X$ and $Y$. Also, the induced map $\mathrm{Res}_{C/R}(\pi):Z\to W$ is finite and \'etale \cite[Proposition 4.9, Proposition 4.10.1]{scheiderer2006real}. The proof is completed by applying Theorem \ref{thm-covering-map} to $\mathrm{Res}_{C/R}(\pi):Z\to W$.
	\end{proof}		
	
	Conversely, it is not true that every finite morphism of reduced complex algebraic varieties inducing a topological covering is \'etale (e.g.\ the normalization of the cuspidal curve). However, one can show that for a finite morphism $f:X\to Y$ of reduced $\mathbb{C}$-varieties, if the base $Y$ is normal and $f_\mathbb{C}$ is a topological covering, then $f$ must be \'etale. To see this, one first passes to the analytification $f^{\mathrm{an}}: X^{\mathrm{an}} \to Y^{\mathrm{an}}$ and by GAGA this is a finite holomorphic map of reduced complex analytic spaces with a normal base, which is locally a homeomorphism \cite[Expos\'e XII]{grothendieck1971revetements}. Then by a result of Grauert-Remmert \cite[Page 166]{grauert1984coherent}, this is actually locally biholomorphic, i.e.\ \'etale. By GAGA again, $f$ itself is \'etale. %
	
	We now return to morphisms over the real closed field $R$. In	applications, one is often interested not in all of $X(R)$, but in subsets obtained by imposing equations and nonvanishing conditions,	and ultimately sign conditions. The next theorem shows that the	covering property is inherited simultaneously by the Boolean pieces	determined by closed subschemes flat over the base.
	
	\begin{theorem} \label{thm-covering-map-of-subschemes}
		Let $X$, $Y$ be two $R$-varieties. Suppose that $\pi: X\to Y$ is a quasi-finite flat $R$-morphism with locally constant geometric fibers. Let $Z_1,\ldots,Z_m$ be closed subschemes of $X$ and set $U_i=X\backslash Z_i$ for $i=1,\ldots,m$. If $\pi|_{Z_i}:Z_i\hookrightarrow X\to Y$ is flat for $i=1,\ldots,m$, then for arbitrary intersection $W=\bigcap_{\alpha \in \Lambda} U_\alpha \cap \bigcap_{\beta \in \Lambda'} Z_\beta$, the restriction $\pi|_W$ induces a covering map between $W(R)$ and $Y(R)$ in the Euclidean topology.
	\end{theorem}
	\begin{proof}
		By Theorem \ref{thm-reduced-g-etale-is-etale} and Corollary \ref{cor:subvariety-covering}, we may assume that $X,Y,Z_1,\ldots,Z_m$ are all reduced, $\pi$ is finite \'etale, and the closed (open respectively) immersions $Z_i\hookrightarrow X$ ($U_i\hookrightarrow X$ respectively) are finite \'etale. Then the immersion $W\hookrightarrow X$ is finite \'etale, as the product of finite \'etale morphisms is still finite \'etale. Then the composition $\pi|_W:W\hookrightarrow X \to Y$ is also a finite \'etale morphism and the conclusion follows from Theorem \ref{thm-covering-map}.
	\end{proof}
	\begin{remark}
		This generalizes \cite[Theorem 5.14]{chen2023geometric} (which itself is a generalization of \cite[Theorem 5]{collins1975quantifier} and \cite[Theorem 3.3]{arnon1984cylindrical}), using a different strategy (Corollary \ref{cor:subvariety-covering}).
	\end{remark}
	
	\begin{example} \label{example:double-cover-of-nodal-curve}
		This example is taken from \cite[Exercise III.10.6]{hartshorne2013algebraic}.
		Consider $$I=\ideal{y_1-x_1^2+1,y_2-x_1 x_2,x_2^2-(x_1^2-1)^2}, J=I\cap k[y_1,y_2]=\ideal{y_1^2(y_1+1)-y_2^2}.$$ 
		
		Set $X=\spec k[y_1,y_2,x_1,x_2]/I$ and $Y=\spec k[y_1,y_2]/J$. Then $Y$ is a nodal curve. The normalization of $Y$ coincides with the blow-up of $Y$ at its singular locus $(0,0)$. $X$ is the variety gluing two normalizations of $Y$ along the exceptional divisors. See Figure \ref{fig:finite-etale-cover}. %
		
		\begin{figure}[htbp]
			\centering              
			\begin{subfigure}[c]{0.35\textwidth}	\setlength{\belowcaptionskip}{-15pt}
				~
				\vspace{9mm}							
				\includegraphics[width=\linewidth]{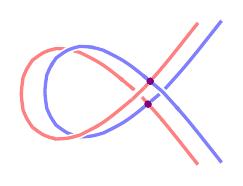}									
				~					
				\caption{Variety $X$}
				\label{fig:double-cover}
			\end{subfigure}
			~
			\begin{subfigure}[c]{0.25\textwidth}	
				~
				\includegraphics[width=\linewidth]{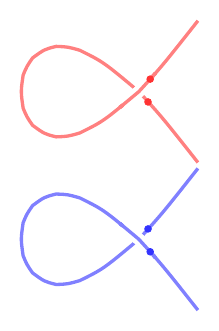}
				~					
				\caption{Two copies of normalization of $Y$} 
				\label{fig:double-cover-irr}
			\end{subfigure}
			~
			\begin{subfigure}[c]{0.35\textwidth} \setlength{\belowcaptionskip}{-9pt}
				~
				\vspace{7mm}
				\includegraphics[width=\linewidth]{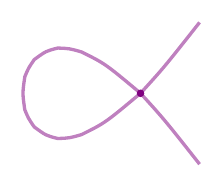}
				~					
				\caption{The nodal curve $Y$} 
				\label{fig:nodal-curve}
			\end{subfigure}
			\caption{A finite \'etale double cover of the nodal curve}\label{fig:finite-etale-cover}
		\end{figure}
		
		Later we shall see that the projection from $X$ to $Y$ is indeed a quasi-finite, flat morphism with locally constant geometric fibers in Example \ref{example:double-cover-of-nodal-curve-verification}. As a result, the projection induces a covering map on the real points by Theorem \ref{thm-covering-map}. Note $X$ and $Y$ are both singular, so results based on the theory of smooth manifolds like Lazard-Rouillier (Corollary \ref{cor-disc-var}) do not apply here.
		
		The last generator of $I$, $x_2^2-(x_1^2-1)^2$ can be factorized as $(x_2-x_1^2+1)(x_2+x_1^2-1)$. Each factor defines an irreducible component of $X$. Choose $x_2-x_1^2+1$ and let $Z$ be the corresponding irreducible component of $X$. Then the projection from $Z$ to $Y$ is the blow-up map. It is easy to see that the restriction of this projection on the regular points $Y_{\mathrm{reg}}$ is flat. Therefore, $Z\times_Y Y_{\mathrm{reg}} \to Y_{\mathrm{reg}}$ is again a finite flat morphism with locally constant geometric fibers and it induces covering maps on the real points (Corollary \ref{cor:subvariety-covering} and Theorem \ref{thm-covering-map-of-subschemes}).
		
		Because $Z$ is defined by $x_2-x_1^2+1=0$ on $X$ and $Z\times_Y Y_{\mathrm{reg}} \to Y_{\mathrm{reg}}$ is flat, the complement $U=X\backslash Z$ yields another covering over $Y_{\mathrm{reg}}$. In particular, the semi-algebraic set defined by
		$$\begin{aligned}
			\left(y_1-x_1^2+1=0 \wedge y_2-x_1 x_2=0\wedge x_2^2-(x_1^2-1)^2=0\wedge x_2-x_1^2+1>0\right)\\ \bigwedge \neg \left(y_1=0\wedge y_2=0\right)
		\end{aligned}$$
		is a topological cover over the real points of $Y_{\mathrm{reg}}$.
	\end{example}

Theorem~\ref{thm-covering-map-of-subschemes} still requires the ambient morphism to be quasi-finite and flat with locally constant geometric fibers. These conditions need not hold globally for an arbitrary morphism. Nevertheless, generic flatness and the lower semi-continuity of fiber dimension allow us to divide the base into finitely many strata. On each stratum, either the preceding covering result applies or every fiber is positive-dimensional.

\begin{proposition}[Covering Stratification]\label{cor-triviality}
	Let $\pi:X\to Y$ be an $R$-morphism of $R$-varieties. Let
	$Z_1,\ldots,Z_m$ be closed subschemes of $X$, and set
	$U_i=X\setminus Z_i$. Then there exists a finite stratification
	\[
	Y=\coprod_{i=1}^s Y_i
	\]
	by locally closed reduced subschemes such that, for every $i$, one
	of the following holds:
	\begin{enumerate}
		\item For every pair of subsets
		$\Lambda,\Lambda'\subseteq\{1,\ldots,m\}$, the restriction
		\[
		\left(
		\bigcap_{\alpha\in\Lambda}U_\alpha
		\cap
		\bigcap_{\beta\in\Lambda'}Z_\beta
		\right)\times_Y Y_i
		\longrightarrow Y_i
		\]
		induces a finite covering map on the $R$-rational points.
		\item Every fiber of
		\[
		X\times_Y Y_i\longrightarrow Y_i
		\]
		is positive-dimensional, and hence is not finite.
	\end{enumerate}
\end{proposition}

\begin{proof}
	We first construct a stratification on which $X$ and all the $Z_j$ are flat over the base. Replace $Y$ by $Y_{\mathrm{red}}$, set $X_0=X$ and $X_j=Z_j$ for $1\leq j\leq m$, and, for every $0\leq j\leq m$, apply the generic flatness stratification \cite[\href{https://stacks.math.columbia.edu/tag/0H3Z}{Lemma 0H3Z}]{stacks-project} to
	\[
	X_j\longrightarrow Y.
	\]
	After taking reduced structures, write the resulting stratification as
	\[
	Y=\coprod_{\alpha\in A_j}S_{j,\alpha},
	\]
	where $X_j\times_Y S_{j,\alpha}\to S_{j,\alpha}$ is flat.
	
	We now take a common refinement of these stratifications. It consists of the nonempty scheme-theoretic intersections
	\[
	S_{\alpha_0,\ldots,\alpha_m}
	=
	\left(
	S_{0,\alpha_0}\times_Y\cdots\times_Y S_{m,\alpha_m}
	\right)_{\mathrm{red}}.
	\]
	Denote these strata by $S_1,\ldots,S_t$. Since flatness is preserved under base change,
	\[
	X\times_Y S_i\longrightarrow S_i
	\quad\text{and}\quad
	Z_j\times_Y S_i\longrightarrow S_i
	\]
	are flat for every $i$ and $j$.
	
	We next separate the finite fibers from the positive-dimensional fibers. Fix $S_i$ and consider
	\[
	d_i:S_i\longrightarrow\{-\infty,0,1,\ldots\},
	\qquad
	y\longmapsto\dim (X\times_Y S_i)_y,
	\]
	where the empty fiber has dimension $-\infty$. Since $X\times_Y S_i\to S_i$ is flat and of finite presentation, $d_i$ is lower semi-continuous \cite[\href{https://stacks.math.columbia.edu/tag/0D4H}{Lemma 0D4H}]{stacks-project}. Hence $S_{i,>0}=\{y\in S_i\mid d_i(y)>0\}$ is open in $S_i$, while	$S_{i,0}=S_i\setminus S_{i,>0}$
	is closed. We endow both with their reduced subscheme structures.
	
	On $S_{i,>0}$, every fiber of $X\times_Y S_{i,>0}\longrightarrow S_{i,>0}$ is positive-dimensional. Thus the second alternative holds.
	
	It remains to treat $S_{i,0}$. Every fiber of $X\times_Y S_{i,0}\longrightarrow S_{i,0}$ is finite, where empty fibers are allowed. Therefore this morphism is quasi-finite. By Proposition~\ref{prop-grothendieck}, we may further stratify $S_{i,0}$ so that its geometric fiber cardinality is locally constant on every resulting stratum.
	
	Let $T$ be one of these strata. Flatness of $X$ and all the $Z_j$ is preserved under base change. Hence Theorem~\ref{thm-covering-map-of-subschemes} shows that, for every pair $\Lambda,\Lambda'\subseteq\{1,\ldots,m\}$, the restriction
	\[
	\left(
	\bigcap_{\alpha\in\Lambda}U_\alpha
	\cap
	\bigcap_{\beta\in\Lambda'}Z_\beta
	\right)\times_Y T
	\longrightarrow T
	\]
	induces a covering map on the $R$-rational points (if both $\Lambda,\Lambda'$ are empty, then the intersection is interpreted as $X$). Its fibers are finite because the ambient morphism is quasi-finite. Thus the first alternative holds.
\end{proof}

	The covering stratification is particularly useful when the original morphism is quasi-finite. In that case, the positive-dimensional alternative cannot occur, so every stratum is a covering stratum. The path-lifting property then allows sample points on the base strata to be lifted to sample points meeting every semi-algebraically connected component of the corresponding Boolean pieces of $X(R)$.

	\begin{corollary}
		Let $f:X\to Y$ be a quasi-finite $R$-morphism of $R$-varieties.
		Let $Z_1,\ldots,Z_m$ be closed subschemes of $X$, and set
		$U_j=X\setminus Z_j$. There exists a finite stratification
		$
		Y=\coprod_{i=1}^sY_i
		$
		by locally closed reduced subschemes with the following property.
		
		For every pair $\Lambda,\Lambda'\subseteq\{1,\ldots,m\}$, set
		\[
		W=
		\bigcap_{\alpha\in\Lambda}U_\alpha
		\cap
		\bigcap_{\beta\in\Lambda'}Z_\beta.
		\]
		If $S_i\subseteq Y_i(R)$ meets every semi-algebraically connected
		component of $Y_i(R)$, then
		\[
		\bigcup_{i=1}^s
		\left(W\times_Y Y_i\right)(R)\cap f_R^{-1}(S_i)
		\]
		meets every semi-algebraically connected component of $W(R)$.
		
		In particular, finding sample points on the semi-algebraically
		connected components of $W(R)$ reduces to finding sample points on
		the semi-algebraically connected components of the $Y_i(R)$ and
		solving the finite fibers over these points.
	\end{corollary}
	
	\begin{proof}
		Apply Proposition~\ref{cor-triviality} to $f$ and
		$Z_1,\ldots,Z_m$. Since $f$ is quasi-finite and quasi-finiteness is
		preserved under base change, the positive-dimensional alternative
		in that proposition cannot occur. We therefore obtain a finite
		stratification
		\[
		Y=\coprod_{i=1}^sY_i
		\]
		such that, for every $W$ as in the statement, the morphism
		\[
		f_{W,i}:W_i=W\times_Y Y_i\longrightarrow Y_i
		\]
		induces a finite covering map on the $R$-rational points.
		
		We first reduce the problem to the individual strata. Since
		\[
		W(R)=\coprod_{i=1}^s W_i(R),
		\]
		a set meeting every semi-algebraically connected component of every
		$W_i(R)$ also meets every semi-algebraically connected component of
		$W(R)$.
		
		Fix $i$, and let $C$ be a semi-algebraically connected component of
		$W_i(R)$. Choose $x_0\in C$ and set $y_0=f_{W,i}(x_0)$. Since
		$S_i$ meets every semi-algebraically connected component of
		$Y_i(R)$, there exists $y_1\in S_i$ lying in the same
		semi-algebraically connected component of $Y_i(R)$ as $y_0$.
		
		Every semi-algebraically connected semi-algebraic set is
		semi-algebraically path-connected
		\cite[Proposition 2.5.13]{bochnak2013real}. Hence there exists a
		semi-algebraic path
		\[
		\gamma:[0,1]\longrightarrow Y_i(R)
		\]
		with $\gamma(0)=y_0$ and $\gamma(1)=y_1$. By the unique path-lifting
		property
		\cite[Proposition 5.7]{delfs1984introduction}
		(or \cite[Theorem 3.3]{delfs1981semialgebraic}), there exists a
		semi-algebraic lift
		\[
		\widetilde{\gamma}:[0,1]\longrightarrow W_i(R)
		\]
		such that
		\[
		\widetilde{\gamma}(0)=x_0,
		\qquad
		f_{W,i}\circ\widetilde{\gamma}=\gamma.
		\]
		The path $\widetilde{\gamma}$ lies in $C$, while
		\[
		\widetilde{\gamma}(1)
		\in W_i(R)\cap f_R^{-1}(y_1).
		\]
		Thus the lifted fibers over $S_i$ meet $C$, completing the proof.
	\end{proof}
	
	The preceding stratification is geometric rather than computational: it does not yet provide equations for the covering strata. To make the construction effective in the affine case, we now compute the loci where finiteness, flatness, and étaleness fail.
	
	\section{Effective Aspect of the Criterion} \label{sect:effective}
	In this section, we are going to show that the condition in Theorem \ref{thm-covering-map} can be effectively verified, at least in the affine case. We assume some basic knowledge about Gr\"{o}bner basis in this section and we use $\mathtt{LC},\mathtt{LM},\mathtt{LT}$ to denote the leading coefficient, leading monomial and leading term of a polynomial. As there are algorithms to compute the radical ideal \cite{krick1991algorithm,eisenbud1992direct,laplagne2006algorithm}, it suffices to check finiteness and \'etaleness for the reduced map by Corollary \ref{cor:equiv-characterization-of-finite-etale}. 
	
	\subsection{Computational Settings}
	Assume that $k$ is an effective field where Gr\"obner basis computation is available. The setting in this section is as follows:
	
	Let $I\unlhd k[y_1,\ldots,y_m,x_1,\ldots,x_n]$ and $J=I\cap k[y_1,\ldots,y_m]$. Then 
	$$B= k[y_1,\ldots,y_m,x_1,\ldots,x_n]/I \text{ and } A=k[y_1,\ldots,y_m]/J$$ are $k$-algebras of finite type and $B$ is also an $A$-algebra. Let $T_x$ ($T_y$ resp.) be a monomial order in $x_i$'s ($y_i$'s resp.). Let $\mathcal{G}$ be a Gr\"{o}bner basis of $I$ in the global polynomial ring $k[y_1,\ldots,y_m,x_1,\ldots,x_n]$ with respect to the block order $T=T_x>T_y$. Additionally, let $\mathcal{G}_y=\mathcal{G} \cap k[y_1,\ldots,y_m]$ and $\mathcal{G}_x=\mathcal{G}\backslash \mathcal{G}_y$, so $\mathcal{G}_y$ is a Gr\"obner basis for $J$ and $\mathcal{G}_x$ consists of the defining equations for $B$ over $A$.
	
	\subsection{Non-finite Locus and Finiteness Check}
	First, we check the finiteness of the map. 
	
	Finiteness is a local property on the base. Therefore there is a biggest open subset of the base on which the restriction is finite. We will show that this set can be immediately read from the block order Gr\"{o}bner basis $\mathcal{G}$ in the affine case. As a result, the finiteness of a morphism can be effectively checked by the emptiness of the locus.
	
	\begin{theorem}[Non-finite Locus] \label{thm:non-finite-locus} The non-finite locus of $\spec B\to \spec A$ can be characterized in the following way:
		\begin{enumerate}
			\item The set $$J_i=\mathtt{LC}_{x_i}(I\cap k[y_1,\ldots,y_m,x_i])=\left\{\mathtt{LC}_{x_i}(f)\middle| f\in I\cap k[y_1,\ldots,y_m,x_i] \right\}$$ is an ideal in $k[y_1,\ldots,y_m]$.
			\item Let $W_i=V\left(J_i\right).$ Then $W=\bigcup_{i=1}^n W_i$ is the smallest closed subset of $\spec A$ such that the restriction $$\spec B\times_{\spec A} \left(\spec A\backslash W\right) \to \left(\spec A\backslash W\right)$$ is finite.
			\item Moreover, $J_i$ is generated by 
			$$\mathcal{G}_i=\left\{\mathtt{LC}_{T_x}(g)\middle|g\in \mathcal{G}, \mathtt{LM}_{T_x}(g)=x_i^{\alpha} \text{ for some }\alpha\geq 0\right\},$$
			the leading coefficients of members of Gr\"{o}bner basis $\mathcal{G}$ whose leading monomial is of the form $x_i^\alpha$. Also, $\mathcal{G}_i$ is a Gr\"{o}bner basis for $J_i$ with respect to $T_y$.
		\end{enumerate}
	\end{theorem}
	\begin{proof}~		
		
		\begin{enumerate}
			\item Trivial.
			\item Let $D(h)=\spec A_h \subseteq \spec A$ be a distinguished open defined by $h\ne 0$. Notice that 			
			$$\begin{array}{cl}
				& D(h) \cap W=\varnothing \\
				\Leftrightarrow& V(h) \supseteq W \\
				\Leftrightarrow& \exists s>0 \text{ s.t.\ } h^s\in \bigcap_{i=1}^n \mathtt{LC}_{x_i}(I\cap k[y_1,\ldots,y_m,x_i]) \\
				\Leftrightarrow& \exists s>0 \text{ s.t.\ } \bigwedge_{i=1}^n \exists f_i\in I\cap k[y_1,\ldots,y_m,x_i] \text{ s.t.\ }\mathtt{LC}_{x_i}(f_i)=h^s \\
				\Leftrightarrow& \bigwedge_{i=1}^n \exists f_i\in I_h\cap k[y_1,\ldots,y_m,x_i]_h \text{ s.t.\ }\mathtt{LC}_{x_i}(f_i)=1\\
				\Leftrightarrow& x_1,\ldots,x_n \text{ are integral over } A_h
			\end{array}$$
			
			Recall that $B_h$ is finite if and only if $x_1,\ldots,x_n$ are integral over $A_h$ \cite[Corollary 4.5]{eisenbud2013commutative}. Therefore, $B_h$ is finite over $A_h$ if and only if $D(h)\cap W=\varnothing$. So $W\subseteq \spec A$ is the smallest closed subset such that $$\spec B\times_{\spec A} \left(\spec A\backslash W\right) \to \left(\spec A\backslash W\right)$$ is finite.
			\item 
			Let $f\in I\cap k[y_1,\ldots,y_m,x_i]$ and set $x_i^\beta:=\mathtt{LM}_{T_x}(f)$. Because $\mathcal{G}$ is a block-order Gr\"{o}bner of $I$ over $k$, we can rewrite 
			$f=\sum\limits_{g\in \mathcal{G}} c_g g$ such that $\mathtt{LM}_T(f)=\max\limits_{g\in\mathcal{G}}\mathtt{LM}_{T}(c_g)\mathtt{LM}_{T}(g)$ by polynomial division algorithm \cite[Theorem 2.3.3]{cox2013ideals}. Since $T$ is the block order $T_x>T_y$, we have $x_i^\beta=\max\limits_{g\in\mathcal{G}}\mathtt{LM}_{T_x}(c_g)\mathtt{LM}_{T_x}(g)$.
			
			Extracting the leading terms in $T_x$, we see that	$$\begin{array}{lcl}
				\mathtt{LC}_{T_x}(f) x_i^\beta &=& \mathtt{LT}_{T_x}(f)\\ &=& \sum\limits_{\substack{g\in \mathcal{G},\\ x_i^\beta=\mathtt{LM}_{T_x}(c_g)\mathtt{LM}_{T_x}(g)}} \mathtt{LT}_{T_x}(c_g) \mathtt{LT}_{T_x}(g) \\
				&=& \sum\limits_{\substack{g\in \mathcal{G},\\ x_i^\beta=\mathtt{LM}_{T_x}(c_g)\mathtt{LM}_{T_x}(g)}} \mathtt{LC}_{T_x}(c_g) \mathtt{LC}_{T_x}(g) \mathtt{LM}_{T_x}(c_g) \mathtt{LM}_{T_x}(g) \\
				&=&\sum\limits_{\substack{g\in \mathcal{G},\\ x_i^\beta=\mathtt{LM}_{T_x}(c_g)\mathtt{LM}_{T_x}(g)}} \mathtt{LC}_{T_x}(c_g) \mathtt{LC}_{T_x}(g) x_i^\beta. \\
			\end{array}$$
			
			Canceling $x_i^\beta$ from both sides, we conclude that $\mathtt{LC}_{T_x}(f)$ is an element of $\ideal{\mathcal{G}_i}=\ideal{\mathtt{LC}_{T_x}(g)|g\in \mathcal{G}, \mathtt{LM}_{T_x}(g)=x_i^{\alpha} \text{ for some }\alpha\geq 0}$: $$\mathtt{LC}_{T_x}(f)=\sum\limits_{\substack{g\in \mathcal{G},\\ x_i^\beta=\mathtt{LM}_{T_x}(c_g)\mathtt{LM}_{T_x}(g)}} \mathtt{LC}_{T_x}(c_g) \mathtt{LC}_{T_x}(g).$$ 
			
			Further extracting the leading terms in $T_y$, we see that $\mathtt{LT}_{T_y}(\mathtt{LC}_{T_x}(f))$ lies in the ideal spanned by leading terms of $\mathcal{G}_i$. So $\mathcal{G}_i$ is a Gr\"{o}bner basis.
		\end{enumerate}
	\end{proof}
	
	By Theorem \ref{thm:non-finite-locus}, Algorithm \ref{algo-non-finite-locus} can be used to compute the non-finite locus of a morphism.
	
	\begin{algorithm}[htbp]
		\KwIn{A block-order Gr\"obner basis $\mathcal{G}$ for $I\unlhd k[y_1,\ldots,y_m,x_1,\ldots,x_n]$.}
		\KwOut{A defining ideal for the smallest closed subset $W$ of $\spec A$ such that the restriction on $\spec A\backslash W$ is finite.}
		\caption{\texttt{NonFiniteLocus}}\label{algo-non-finite-locus}
		
		$J_1,\ldots,J_n:=J$\;
		\ForEach{$g\in \mathcal{G}_x$}{
			\If{$\mathtt{LM}_{T_x}(g)$ is of the form $x_i^\alpha$}{
				$J_i:=J_i+\ideal{\mathtt{LC}_{T_x}(g)}$\;
			}
		}
		
		\Return{$\prod_{i=1}^{n} J_i$}\;
	\end{algorithm}
	
	As a corollary, whether a morphism is finite or not can be effectively decided. The following statement slightly improves the finiteness check in \cite[Proposition 3.1.5]{greuel2008singular} by allowing a more general choice of monomial order $T$.
	
	\begin{corollary}[Finiteness Test, c.f.\ {\cite[Proposition 3.1.5]{greuel2008singular}}]\label{cor-finiteness-test}
		The finiteness of morphism $\spec B\to \spec A$ can be effectively determined. Especially, the following are equivalent:
		\begin{enumerate}
			\item $B$ is a finite $A$-module. \label{item-finite}
			\item There are monic polynomials $\varphi_i(\lambda)\in A[\lambda]$ such that $\varphi_i(x_i)=0$ in $B$. \label{item-integral}
			\item For each $x_i$, there is $g_i\in \mathcal{G}$ such that the leading monomial $\mathtt{LM}_T(g_i)$ is a power of $x_i$. \label{item-leading-term}
		\end{enumerate}
	\end{corollary}
	\begin{proof}~
		By Theorem \ref{thm:non-finite-locus}, $B$ is finite over $A$ if and only if the ideals $J_i=\ideal{1}$, or equivalently $1\in\mathcal{G}_i$ for $i=1,\ldots,n$.
	\end{proof}
	\begin{remark}
		We remark that the finiteness check has been studied under a different notion called ``properness defect'' \cite{jelonek1999testing,stasica2002effective,schost2004properness,lazard2007solving,moroz2011properness}. Of course, for affine morphisms, being proper is equivalent to being finite \cite[Exercise II.4.6]{hartshorne2013algebraic}.
	\end{remark}

	\begin{remark}	
		A common feature of previous results on detecting ``properness defect'' is the passage to the projective space \cite{stasica2002effective,lazard2007solving,moroz2011properness}. This is not always convenient, because of the necessary use of graded reverse lexicographic order (the input might be a Gr\"{o}bner basis with respect to another order), and extra variables are explicitly or implicitly introduced during the computation. Moreover, both \cite{stasica2002effective} and \cite[Algorithm 1]{moroz2011properness} require multiple eliminations. Our method needs only one block-order Gr\"{o}bner basis computation without introducing new variables. Moreover, our method does not make any assumption on the base field and the affine variety.
		
	\end{remark}	
	
	\subsection{Non-finite-flat Locus and Flatness Check}
	We recall that the flatness of a finitely presented module can be checked via the Fitting ideals.
	\begin{lemma}[Flatness Test, {\cite[Corollary 7.3.13]{greuel2008singular}}]\label{lem-flatness-test}
		Suppose $B$ is a finite $A$-module with a finite presentation
		$$A^s\mathop{\to}\limits^{N} A^r\to B\to 0,$$
		where $N\in A^{r\times s}$.
		Define the $i$-th Fitting ideal of $B$, $F_{i}(B)\unlhd A$ to be ideal generated by all $(r-i)$-minors of $N$ and set $F_{-1}(B)=\ideal{0}$. Define $$F(B)=\sum_{i=0}^{r}\left((0:F_{i-1}(B))\cap F_i(B)\right).$$
		Then, \begin{enumerate}
			\item $B$ is flat over $A$ if and only if $1\in F(B)$.
			\item Let $p\in \spec A$, then $B_p$ is flat over $A_p$ if and only if $p\notin V(F(B))$.
		\end{enumerate}			
	\end{lemma}
	
	The idea of flatness test is that flatness can be checked locally. Namely, $B_p$ is a free (equivalently, flat) $A_p$-module of rank $i$, if and only if $F_i(B_p)=A_p$ and $F_{i-1}(B_p)=0$, i.e.\ $p$ lies outside the zero locus of $F_i(B)$ and the support of $F_{i-1}(B)$, which is the zero locus of the annihilator $\mathop{\mathrm{Ann}}(F_{i-1}(B))=(0:F_{i-1}(B))$.
	
	To check the flatness of an $A$-algebra $B$ via the above criterion, we need a finite presentation of it as an $A$-module. However, we have already seen that $B$ may fail to be finite over $A$ in a closed subset in the previous section. An obvious solution is to cover the finite locus by affine opens and apply the criterion to each affine open. But then one has to do the test multiple times, which can be computationally very expensive. Here we propose a direct approach without localization.

Recall that the non-finite locus is characterized by the leading coefficients in $x_i$. For each $i$, let
\[
\mathcal P_i=
\left\{g\in\mathcal G_x\middle|
\mathtt{LM}_{T_x}(g)=x_i^{d(g)}\text{ for some }d(g)>0\right\}.
\]
If some $\mathcal P_i$ is empty, the finite locus is empty. Otherwise, set
\[
t_i=\max\{d(g)\mid g\in\mathcal P_i\}
\]
and consider the monomials
\begin{equation}\label{eqn:uniform-monomials}\tag{$\star$}
	\mathcal M=
	\left\{
	\prod_{i=1}^n x_i^{a_i}
	\middle|
	0\leq a_i\leq t_i-1\text{ for }i=1,\ldots,n
	\right\}.
\end{equation}
Let $B'\subseteq B$ be the $A$-submodule generated by the residue classes of the monomials in $\mathcal M$. The module $B'$ is finite over $A$, although it need not be equal to $B$ over the non-finite locus. The flatness test (Lemma \ref{lem-flatness-test}) applies to $B'$ once we have a finite presentation. And we will show that the flatness of $B'$ is equivalent to the flatness of $B$ outside the non-finite locus.

\begin{lemma}\label{lem:uniform-monomial-span}
	Let $W\subseteq\spec A$ be the non-finite locus in
	Theorem~\ref{thm:non-finite-locus}. Then
	\[
	B'_p=B_p
	\qquad\text{for every }p\in\spec A\setminus W.
	\]
\end{lemma}

\begin{proof}
	We show that every $x$-monomial can be reduced to an $A_p$-linear combination of monomials in $\mathcal M$. Let $p\in\spec A\setminus W$. For every $1\leq i \leq n$, there is some $g_i\in\mathcal P_i$ such that $\mathtt{LC}_{T_x}(g_i)+J \notin p$.

	Thus $\mathtt{LC}_{T_x}(g_i)$ is a unit in $A_p$. Hence $g_i=0$ in $B_p$ gives a rewriting rule for $x_i^{d(g_i)}$. Suppose $B'_p\neq B_p$, then there is a smallest monomial $\mathfrak{m}=\prod_i x_i^{\alpha_i}$ with respect to $T_x$ that does not belong to $B'_p$. Clearly by this assumption, $\mathfrak{m} \notin \mathcal{M}$. Without loss of generality, we may assume that $\alpha_1\geq t_1\geq \deg_{x_1} g_1$. Then we may rewrite $\mathfrak{m}$ as an $A_p$-linear combination of smaller monomials using $g_1=0$. But those monomials belong to $B'_p$ by the minimality of $\mathfrak{m}$. Hence $\mathfrak{m}\in B'_p$. This is a contradiction and we are done.
\end{proof}

It remains to compute a presentation of $B'$. Fix an ordering
$\mathcal M=\{\mathfrak{m}_1,\ldots,\mathfrak{m}_r\}$, introduce variables
$t,u_1,\ldots,u_r$, and form the graph ideal
\[
\mathcal L
=I+\ideal{u_j-t\mathfrak{m}_j\mid j=1,\ldots,r}
\unlhd k[y_1,\ldots,y_m,x_1,\ldots,x_n,t,u_1,\ldots,u_r].
\]
Notice that $\mathcal L$ is homogeneous in $u_1,\ldots,u_r, t$. Compute a Gr\"obner basis
${\mathcal H}$ of $\mathcal L$ with respect to a block order
\[
[\bs x,t]>[\bs u]>[\bs y]
\]
such that the monomial order in $\bs u$ is degree-compatible, and let
\[
\mathcal H_1=
\left\{
h\in{\mathcal H}\cap k[y_1,\ldots,y_m,u_1,\ldots,u_r]
\middle|\deg_{\bs u}(h)=1
\right\}.
\]

\begin{lemma}[Relations among the monomials]
	\label{lem:relations-among-monomials}
	Write $\mathcal H_1=\{h_1,\ldots,h_s\}$ and
	\[
	h_j=\sum_{i=1}^r n_{ij}(\bs y)u_i.
	\]
	The matrix
	$N=(n_{ij})\in A^{r\times s}$
	gives a finite presentation
	\[
	A^s\mathop{\longrightarrow}\limits^N A^r
	\longrightarrow B'\longrightarrow0.
	\]
\end{lemma}

\begin{proof}
		
	Consider
	$$\begin{array}{rcl}
	\theta:A[u_1,\ldots,u_r]&\to& B[t]\\
	u_i &\mapsto & \mathfrak{m}_i t.
	\end{array}$$	
	A $\bs{u}$-linear form $\sum_i a_i(\bs y)u_i$ belongs to $\ker \theta$ if and only if	$\sum_i a_i(\bs y)\mathfrak{m}_i=0\text{ in }B.$
	
	Since $k[\bs{y},\bs{x},\bs{u},t]/\mathcal L\simeq B[t]$, by the elimination theorem we have the preimage of $\ker\theta$ in $k[y_1,\ldots,y_m,u_1,\ldots,u_r]$ is exactly $\mathcal L\cap k[y_1,\ldots,y_m,u_1,\ldots,u_r]$.
	
	Thus the degree-one part of the	elimination ideal consists exactly of the $A$-linear relations among the monomials in $\mathcal M$. As $\mathcal{H}$ is a Gr\"{o}bner basis with respect to a monomial order homogeneous in $\bs{u}$, $\mathcal{H}_1$ spans the degree-one part.
\end{proof}

Let $N$ be the matrix in Lemma~\ref{lem:relations-among-monomials}. Define
ideals by
$$F_i =J+\ideal{\mathtt{minors}(N,r-i)},\ i=-1,0,\ldots,r; $$
and
\[
F
=\sum_{i=0}^r\bigl((J: F_{i-1})\cap F_i\bigr).
\]
The image of $F$ in $A$ is the flatness ideal $F(B')$. As usual, if the requested minors do not exist, they contribute no generators; the unique $0\times0$ minor is $1$.

\begin{theorem}[Non-finite-flat Locus]
	\label{thm:non-finite-flat-locus}
	The largest open of $\spec A$ where the map is finite flat is the complement of 
	$$W \cup V(F),$$
	where $W$ is the non-finite locus in Theorem \ref{thm:non-finite-locus} and $F$ is the Fitting ideal of $B'$.
\end{theorem}

\begin{proof}
	We first compare $B$ with the finite module $B'$ on the finite locus, and
	then apply the Fitting-ideal criterion. By
	Lemma~\ref{lem:uniform-monomial-span}, $B'_p=B_p$ for every
	$p\notin W$. Hence the restriction of $B$ at such a point is flat if and
	only if $p\notin V(F)$ by Lemma \ref{lem-flatness-test} and Lemma \ref{lem:relations-among-monomials}. Therefore the non-finite-flat locus is
	$W\cup V(F)$
	as claimed.
\end{proof}

Algorithm \ref{algo:non-finite-flat-locus} computes the non-finite-flat locus.
\begin{algorithm}[htbp]
	\KwIn{A block-order Gr\"obner basis $\mathcal G$ for
		$I\unlhd k[y_1,\ldots,y_m,x_1,\ldots,x_n]$.}
	\KwOut{An ideal defining the non-finite-flat locus in $\spec A$.}
	\caption{\texttt{NonFiniteFlatLocus}}
	\label{algo:non-finite-flat-locus}
	
	\For{$i:=1$ \KwTo $n$}{
		$\mathcal P_i:=\{g\in\mathcal G_x\mid
		\mathtt{LM}_{T_x}(g)=x_i^{d(g)},\ d(g)>0\}$\;
		\If{$\mathcal P_i=\varnothing$}{\Return{$J$}\;}
		$J_i:=J+\ideal{\mathtt{LC}_{T_x}(g)\mid g\in\mathcal P_i}$\;
		$t_i:=\max\{d(g)\mid g\in\mathcal P_i\}$\;
	}
	$\mathcal{M}:=\left\{\prod_{i=1}^n x_i^{a_i}\middle|0\leq a_i\leq t_i-1\text{ for }i=1,\ldots,n \right\}=\{\mathfrak{m}_1,\ldots,\mathfrak{m}_r\}$\;
	Introduce fresh variables $t,u_1,\ldots,u_r$\;
	$\mathcal L:=I+\ideal{u_j-t\mathfrak{m}_j\mid j=1,\ldots,r}$\;
	Compute a Gr\"obner basis ${\mathcal H}$ of
	$\mathcal L$ for $[\bs x,t]>[\bs u]>[\bs y]$, $\bs{u}$-degree-compatible\;
	$\mathcal H_1:=\{h\in{\mathcal H}\cap k[\bs{y},\bs u]\mid
	\deg_{\bs u}(h)=1\}=\{h_1,\ldots,h_s\}$\;
	Let $N$ be the $r\times s$ matrix whose $(i,j)$-entries are $u_i$-coefficient of $h_j$\;
	$F_{-1}:=J$\;
	\For{$i:=0$ \KwTo $r$}{
		$F_i:=J+\ideal{\mathtt{minors}(N,r-i)}$\;
	}
	$F:=\sum_{i=0}^r\bigl((J: F_{i-1})\cap F_i\bigr)$\;
	\Return{$J_1 \cdot J_2\cdot \cdots \cdot J_n \cdot F$}\;
\end{algorithm}

\begin{corollary}[Finite Flatness Test]
	\label{cor:finite-flat-test}
	Whether the morphism $\spec B\to \spec A$ is finite flat can be effectively determined by Algorithm \ref{algo:non-finite-flat-locus}.
\end{corollary}

	\begin{remark}\label{rmk:miracle-flatness}
		The flatness check can be skipped if the conditions for miracle flatness hold. For example, if the morphism $f$ is a finite endomorphism $f:X\to X$ with $X$ integral regular, then $f$ is automatically flat. This is often the case in arithmetic dynamics.
	\end{remark}
	
	\subsection{Non-\'etale Locus and \'Etaleness Check}
	Finally, we now compute the largest open subset of $\spec A$ over which the morphism
	$\spec B\to\spec A$ is finite \'etale. In our setting, a morphism is finite	\'etale if and only if it is finite flat and unramified. Let $W\cup V(F)\subseteq\spec A$ be the non-finite-flat locus given by Theorem \ref{thm:non-finite-flat-locus},	and put $U=\spec A\setminus (W\cup V(F))$. It remains to determine where the restriction over $U$ fails to be unramified.
	
	The ramification locus in $\spec B$ is the support of the module of relative K\"ahler differentials $\Omega_{B/A}$. Equivalently, it is the zero locus of the zeroth Fitting ideal of $\Omega_{B/A}$, which can be computed from the relative Jacobian matrix. Since the restriction
	\[
	\spec B\times_{\spec A}U\to U
	\]
	is finite, the image of this closed locus is closed in $U$.
	
	\begin{theorem}[Non-finite-\'etale Locus]\label{thm:non-finite-etale-locus}
		Suppose $\{g_1,\ldots,g_t\}$ is a set of generators of $I$.
		Define $$\mathcal{J}=\begin{pmatrix}
			\frac{\partial g_1}{\partial x_1}& \ldots & \ldots & \frac{\partial g_t}{\partial x_1} \\
			\vdots & &  &\vdots \\
			\frac{\partial g_1}{\partial x_n}& \ldots & \ldots & \frac{\partial g_t}{\partial x_n} \\
		\end{pmatrix}$$
		and $\mathop{\mathrm{Jac}}(B)$ is the ideal generated by all $n$-minors of $\mathcal{J}$ plus $I$.
		
		The largest open of $\spec A$ where the map is finite \'etale is the complement of 
		$$W\cup V(F)\cup V(\mathrm{Jac}(B)\cap k[y_1,\ldots,y_m]),$$
		where $W$ is the non-finite locus in Theorem \ref{thm:non-finite-locus}, $F$ is the Fitting ideal of $B'$ in Theorem \ref{thm:non-finite-flat-locus}.
	\end{theorem}
	\begin{proof}
		By Theorem \ref{thm:non-finite-locus} and Theorem \ref{thm:non-finite-flat-locus}, the non-finite-flat locus is given by $W\cup V(F)$. Since the ramification locus can be found by the Jacobian criterion, elimination gives the closure of the image of the ramification locus. Over $\spec A\setminus W$, the morphism is finite and hence closed. Therefore taking the closure introduces additional points only inside $W$.
	\end{proof}
	
	Algorithm \ref{algo-non-finite-etale-locus} turns Theorem \ref{thm:non-finite-etale-locus} into effective computations.
	
	\begin{algorithm}[htbp]
		\KwIn{A block-order Gr\"obner basis $\mathcal G$ for
			$I\unlhd k[y_1,\ldots,y_m,x_1,\ldots,x_n]$.}
		\KwOut{An ideal defining the non-finite-\'etale locus in $\spec A$.}
		\caption{\texttt{NonFiniteEtaleLocus}}
		\label{algo-non-finite-etale-locus}
		
		$N
		:=\mathtt{NonFiniteFlatLocus}(\mathcal G)$\;
		
		Let $\mathcal{J}$ be the matrix $\left(\frac{\partial g}{\partial x_i}\right)_{i=1,\ldots,n,g\in\mathcal{G}_x}$\;
		
		$\mathrm{Jac}:=I+\ideal{\mathtt{minors}(\mathcal J,n)}$\;
		
		Compute a Gr\"obner basis $\mathcal H$ of $\mathrm{Jac}$ with respect to
		the block order $T_x>T_y$\;
		
		$\mathcal{H}_y:=\mathcal H\cap k[y_1,\ldots,y_m]$\;
		
		\Return{$N\cdot \ideal{\mathcal{H}_y}$}\;
	\end{algorithm}

	\begin{corollary}[Finite \'Etaleness Test]
		\label{cor:finite-etale-test}
		Whether the morphism $\spec B\to \spec A$ is finite \'etale can be effectively determined. In particular, the conditions in Theorem \ref{thm-covering-map} can be effectively verified.
	\end{corollary}
	\begin{remark}
		If one is only concerned with deciding whether
		$\spec B\to\spec A$ is finite \'etale, then the computation can be
		simplified considerably.
		
		First, there is no need to eliminate the $x$-variables from the Jacobian ideal. Indeed, the morphism is unramified if and only if $\mathrm{Jac}=\ideal{1}$. Thus, after the finiteness test, one may	compute $\mathrm{Jac}$ before carrying out the more expensive flatness test. If $\mathrm{Jac}$ is a proper ideal, then the morphism cannot be finite \'etale and the computation terminates immediately.
		
		More generally, the computation may terminate as soon as one of the	necessary conditions fails. If some coefficient ideal $J_i$ is proper, then the morphism is not finite. If $\mathrm{Jac}$ is proper, then it is not unramified. Finally, if the flatness ideal $F$ is proper, then it is not flat. Only when all three tests succeed is the morphism finite \'etale.
		
		We emphasize that the individual Fitting ideals $F_i$ need not be unit;	flatness is determined by the combined ideal
		\[
		F=\sum_{i=0}^r\bigl((J:F_{i-1})\cap F_i\bigr).
		\]
		Under additional hypotheses guaranteeing flatness, such as those in
		Remark \ref{rmk:miracle-flatness}, the flatness computation may be omitted altogether.		
	\end{remark}

	\begin{example} \label{example:double-cover-of-nodal-curve-verification}
		
		Recall the double cover $X$ of nodal curve $Y$ in Example \ref{example:double-cover-of-nodal-curve}. Their defining ideals are $$I=\ideal{y_1-x_1^2+1,y_2-x_1 x_2,x_2^2-(x_1^2-1)^2}, J=I\cap k[y_1,y_2]=\ideal{y_1^2(y_1+1)-y_2^2}.$$
		We will verify that the projection from $X$ to $Y$ is finite \'etale using Algorithm \ref{algo-non-finite-etale-locus}.
		
		Using a block order $\mathtt{grevlex}(x_1,x_2)>\mathtt{grevlex}(y_1,y_2)$, we compute a Gr\"obner basis for $I$ (with leading terms marked):
		$$\mathcal{G}=\left\{\underline{\mathstrut y_1^3}+y_1^2-y_2^2,\underline{\mathstrut y_2 x_1}-y_1 x_2-x_2, \underline{\mathstrut y_1^2 x_1}- y_2 x_2 ,\underline{\mathstrut x_2^2}-y_1^2,\underline{\mathstrut x_1 x_2}-y_2,\underline{\mathstrut x_1^2}-y_1-1\right\}.$$
		
		As $x_1^2-y_1-1, x_2^2-y_1^2\in \mathcal{G}$, we conclude that the projection is finite. The spanning monomials can be chosen to be $\mathcal{M}=\{1, x_1, x_2, x_1 x_2\}$. By Lemma \ref{lem:relations-among-monomials}, the relations are given by $$\mathcal{N}=\left\{-y_1 x_2+ y_2 x_1-x_2, y_1^2 x_1- y_2 x_2 ,x_1 x_2-y_2\right\}.$$
		Then the presentation matrix is
		$$N=\left(
		\begin{array}{cccc}
			0 & y_2 & -y_1-1 & 0 \\
			0 & y_1^2 & -y_2 & 0 \\
			-y_2 & 0 & 0 & 1 \\
		\end{array}
		\right)^{\top}$$
		Expanding the minors of $N$, we see that the Fitting ideals are
		$$F_0=F_1=\ideal{y_1^3+y_1^2-y_2^2}=J, F_2=F_3=F_4=\ideal{1}.$$
		Therefore the projection is also flat.
		
		Finally we check the smoothness of fibers. The Jacobian matrix is $$\mathcal{J}=\left(
		\begin{array}{ccc}
			2 x_1 & x_2 & 4 x_1^3-4 x_1 \\
			0 & x_1 & -2 x_2
		\end{array}
		\right).$$
		So the Jacobian ideal is $I$ plus the ideal generated by all $2$-minors of $\mathcal{J}$. A direct computation shows that the Jacobian ideal is $\ideal{1}$. Therefore all fibers are smooth and the projection is finite \'etale. 
	\end{example}
	
	\section{Applications}\label{sect-app}
	
	We illustrate the preceding results with two examples.

	\begin{example}\label{example:alg-stats}
		This example is from algebraic statistics {\cite[Example 1.17]{huh2014likelihood}}. A random variable $\xi$ follows an unknown discrete distribution $(p_0,p_1,p_2)$, where the distribution shall satisfy a statistical model:
		$$F(p_0,p_1,p_2)=p_2(p_1-p_2)^2+(p_0-p_2)^3=0.$$
		We can model all possible distributions by $V(F)$ in the projective plane $\mathbb{P}^2_p$. Here the subscript $p$ stands for the probability.
		
		Suppose our observation for $\xi$ is $(u_0,u_1,u_2)$, and we model all observations by another projective plane $\mathbb{P}^2_u$. Our goal is to find the most ``reasonable'' distribution from the observation. This can be done by the method of maximum likelihood estimation (MLE). The family of all estimated distributions for all observations is the \emph{likelihood correspondence}, a closed subvariety $$X\hookrightarrow V(F)\times \mathbb{P}_u^2 \hookrightarrow \mathbb{P}_p^2 \times \mathbb{P}_u^2.$$ By a standard argument, the defining equations for $X$ in $\mathbb{Q}[p_0,p_1,p_2,u_0,u_1,u_2]$ can be found by saturating the ideal $\left\langle{\left|\begin{smallmatrix}
				1 & 1 & 1 \\ u_0 p_1 p_2 & p_0 u_1 p_2 & p_0 p_1 u_2 \\ {\partial F}/{\partial p_0} & {\partial F}/{\partial p_1} & {\partial F}/{\partial p_2}
			\end{smallmatrix}\right| , F }\right\rangle$ with respect to $\ideal{p_0 p_1 p_2 (p_0+p_1+p_2)}\cap \ideal{p_0-p_2,p_1-p_2}$.
		
		The fibers of projection to the observation space $\pi:X\to \mathbb{P}_u^2$ are the MLEs of the observation. The length of the generic fiber is known as the \emph{ML degree} of the statistic model $F$. It is known that the ML degree of $F$ is $5$ \cite{huh2014likelihood}, which means that there are $5$ sets of ``distributions'' $(p_0,p_1,p_2)$ over $\mathbb{C}$ for a general observation data $(u_0,u_1,u_2)$.
		
		However, a complex ``distribution'' does not make much sense, as probabilities $p_0,p_1,p_2$ must be real numbers. We will now scrutinize the real fibers of $\pi_\mathbb{R}$ more carefully. We invite the readers to verify our arguments below using Algorithm \ref{algo-non-finite-locus}-\ref{algo-non-finite-etale-locus} on a computer algebra system.
		
		First, the projection $\pi:X \to \mathbb{P}^2_u$ is already a finite one. A brute-force way to see this is to check that all fibers are finite, as a quasi-finite projective morphism is finite. A more elegant solution is to cover $\mathbb{P}_u^2$ by affine opens $D_+(u_0+u_1+u_2),D_+(u_0+3u_2)$ and $D_+(u_0(u_0+u_1))$. Their preimages are again affine opens $D_+(p_0+p_1+p_2),D_+(p_2)$ and $D_+(p_1)$ respectively, so we are in the affine case where Algorithm \ref{algo-non-finite-locus} is applicable.
		
		Next, one can verify by Algorithm \ref{algo:non-finite-flat-locus} that the non-flat locus of the finite morphism $\pi$ is a single point $(1:1:1)\in \mathbb{P}^2_u$, i.e.\ the restriction of $\pi$ on $\mathbb{P}^2_u\backslash(1:1:1)$ is flat. Therefore, all fibers except the fiber over $(1:1:1)$ of $\pi$ are of the same length \cite[Corollary III.9.10]{hartshorne2013algebraic}. This length is the ML degree of $F$, which is $5$. But the length of fiber over $(1:1:1)$ is $6$. Hence the jump in length is purely due to the non-flatness.
		
		Then, the discriminant locus is a curve $C$ of degree $8$ in $\mathbb{P}_u^2$ (this can be found by Algorithm \ref{algo-non-finite-etale-locus}). This is where the geometric fiber cardinality is different from the expected number $5$. The exceptional point $(1:1:1)$ also lies on $C$. Therefore, the restriction of $\pi$ on $\mathbb{P}_u^2\backslash C$ is a finite flat morphism with constant geometric fibers. By Theorem \ref{thm-covering-map}, $\pi_{\mathbb{R}}$ is a covering map on $\mathbb{P}_u^2(\mathbb{R})\backslash C(\mathbb{R})$, and the number of real fibers is a constant on each connected component of $\mathbb{P}_u^2(\mathbb{R})\backslash C(\mathbb{R})$.
		
		The singular locus of $C$ consists of $21$ points, and $7$ of them are real. The geometric fiber cardinality $n(y)$ is $4$ on the smooth points of $C$, but this number drops to $3$ on $\mathop{\mathrm{Sing}}C$. Therefore, the restriction of $\pi$ on $C\backslash (\mathop{\mathrm{Sing}}C \cup (1:1:1))$ (base change) is again a finite flat morphism with constant geometric fibers and we can apply Theorem \ref{thm-covering-map} again here.
		
		\begin{figure}[hbtp]
			\centering
			\includegraphics[width=0.55\linewidth]{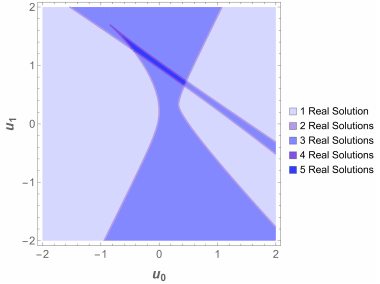}
			\caption{The number of real fibers in $D_+(u_0+u_1+u_2)$.} 
			\label{fig:alg-stats}
		\end{figure}
		
		We plot the regions where the observation has $1,\ldots,5$ distinct real MLEs using different colors in Figure \ref{fig:alg-stats}. One can summarize the result as follows:
		\begin{enumerate}
			\item There is a region in the center of the figure where each observation has $5$ different real MLEs.
			\item Each observation on the boundaries of the aforementioned region has $4$ different real MLEs, except for four corners, which are singularities of the discriminant locus. Each corner point corresponds to three real $(p_0:p_1:p_2)$.			
			\item The region for $3$ real MLEs is connected (the upper, lower and middle-right regions in Figure \ref{fig:alg-stats} are connected at the infinity).
			\item For an observation with a unique real MLE, there are two connected components: the left-hand region and the upper-right region actually lie in the same connected component.
			\item The remaining points are the border between observations with $1$ and $3$ real MLEs. They have two distinct real MLEs.
		\end{enumerate}
		
		One can even require that all the observation data $(u_0:u_1:u_2)$ are positive and look for the number of positive MLEs $(p_0:p_1:p_2)$. It can be shown that when the observation is positive, one of the real solutions is never positive, but the remaining real solutions are. See Figure \ref{fig:alg-stats-pos}.
		
		\begin{figure}[hbtp]
			\centering
			\includegraphics[width=0.55\linewidth]{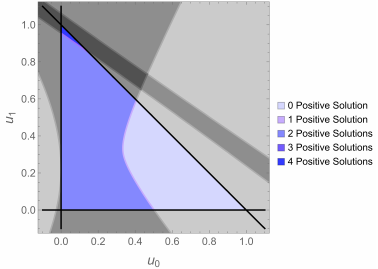}
			\caption{The number of positive fibers in the positive quadrant of $D_+(u_0+u_1+u_2)$.} 
			\label{fig:alg-stats-pos}
		\end{figure}
		
		Previously, \cite{rodriguez2015data,rodriguez2017probabilistic} developed a notion called data-discriminant, which corresponds to the curve $C$ in this example. The main result about data-discriminant locus in \cite{rodriguez2015data,rodriguez2017probabilistic} is that the number of real or positive solutions is uniform over each connected component of the complement of the data-discriminant locus. Using our results, it is even possible to analyze the special data on the data-discriminant locus.
		
	\end{example}

	Finally, we study a matrix completion problem using all the techniques we have developed so far. Here, a matrix $M$ is partially observed and we need to recover it under either a low-rank or positive semi-definite hypothesis on $M$.
	
	\begin{example}[Matrix Completion]
		Consider the symmetric matrix $$\begin{array}{lcc}
			M&=&\left(
			\begin{array}{cccc}
				-v & 0 & u-x & 2 v-x \\
				0 & u+v & -v-x & u+2 v-x \\
				u-x & -v-x & -u+2 v+y & -2 u-v+y \\
				2 v-x & u+2 v-x & -2 u-v+y & y \\
			\end{array}\right)\\
			&=&{\footnotesize u\cdot \begin{pmatrix}						
					0 & 0 & 1 & 0 \\
					0 & 1 & 0 & 1 \\
					1 & 0 & -1 & -2 \\
					0 & 1 & -2 & 0 \\						
				\end{pmatrix}+v \cdot \begin{pmatrix}
					-1 & 0 & 0 & 2 \\
					0 & 1 & -1 & 2 \\
					0 & -1 & 2 & -1 \\
					2 & 2 & -1 & 0 \\
				\end{pmatrix}+
				\begin{pmatrix}
					0 & 0 & -x & -x \\
					0 & 0 & -x & -x \\
					-x & -x & y & y \\
					-x & -x & y & y \\
			\end{pmatrix}}.
		\end{array}
		$$
		Here, the last summand $\left(\begin{smallmatrix}
			0 & 0 & -x & -x \\
			0 & 0 & -x & -x \\
			-x & -x & y & y \\
			-x & -x & y & y \\
		\end{smallmatrix}\right)$ is the observed part of $M$ and $u,v$ are unknown.
		
		\noindent Question: \begin{enumerate}[label*=(\roman*)]
			\item ({\textit{Low-rank completion}}) For which $x,y\in \mathbb{R}$ values can $M$ be completed to a real matrix of rank $2$, i.e.\ there exist $u,v\in \mathbb{R}$ such that $M$ is a real symmetric matrix of rank $2$?
			\item ({\it PSD completion}) For which $x,y\in \mathbb{R}$ values can $M$ be completed to a p.s.d.\ of rank $2$, i.e.\ there exist $u,v\in \mathbb{R}$ such that $M$ is a real symmetric positive semi-definite matrix of rank $2$?
		\end{enumerate}

		Let $I$ be the ideal generated by all $3$-minors of $M$ and let $J$ be the ideal generated by all $2$-minors of $M$. Set $X=\spec \mathbb{R}[x,y,u,v]/I$, $Z=\spec \mathbb{R}[x,y,u,v]/J$ and $Y=\spec \mathbb{R}[x,y]$. Clearly $Z$ is a closed subvariety of $X$, and the real rank-two matrices are the $\mathbb{R}$-rational points of $X\backslash Z$. It can be easily checked that the projection $X\to Y$ is quasi-finite.
		
		By recursively applying generic flatness and counting the geometric fiber, we may stratify $Y$ into a union of locally closed subvarieties $Y_i$ such that the projection $X\times_Y Y_i \to Y_i$ is a quasi-finite, flat morphism with constant geometric fibers. For example, the stratification can be taken as
		$$\left\{\begin{array}{rcl}
			Y_1&=&D(x (x-y) (8 x^4+4 x^2 y^2-4 x y^3+y^4)),\\
			Y_2&=&V(8 x^4+4 x^2 y^2-4 x y^3+y^4)\cap D(x(x-y)),\\
			Y_3&=&V(x)\cap D(x-y),\\
			Y_4&=&V(x-y)\cap D(x),\\
			Y_5&=&V(x,y).
		\end{array} \right.$$					
		
		Similarly $Y$ can be stratified into a union of locally closed subvarieties $Y_j'$ such that the projection $Z\times_Y Y_j' \to Y_j'$ is flat. A possible stratification is
		$$\left\{\begin{array}{rcl}
			Y_1'&=&D(x),\\
			Y_2'&=&V(x)\cap D(y),\\
			Y_3'&=&V(x,y).
		\end{array} \right.$$
		
		Then for arbitrary $Y_i$ and $Y_j'$, we have $X \times_Y Y_i \times_Y Y_j' \to Y_i \times_Y Y_j'$ is a quasi-finite, flat morphism with constant geometric fibers and $Z \times_Y Y_i \times_Y Y_j' \to Y_i \times_Y Y_j'$, the restriction of projection on $Z$, is flat. By Theorem \ref{thm-covering-map-of-subschemes}, the projection of $X(\mathbb{R})\backslash Z(\mathbb{R})$ on $Y_i(\mathbb{R}) \cap Y_j'(\mathbb{R})$ is a covering map.
		
		The real affine plane $\mathbb{R}^2$ can be decomposed into $9$ semi-algebraically connected sets $C_1,\ldots,C_9$, depending on the sign condition of $x$ and $y-x$, such that for all $Y_i$, $Y_j'$, either $C_k\cap Y_i(\mathbb{R}) \cap Y_j'(\mathbb{R})=\varnothing$ or $C_k\subseteq Y_i(\mathbb{R}) \cap Y_j'(\mathbb{R})$. %
		
		Then for each cell $C_k$, the number of real fibers of $X\backslash Z$ is a constant. Therefore we can choose one point from each cell $C_k$ to check whether $M$ can be completed to a real symmetric matrix of rank $2$ for all $(x,y)\in C_k$. For example, let us choose $(x,y)=(1,1)$ from $\{(x,y)|x>0,y=x\}$ and plug it into $M$. We get 
		$$\left(
		\begin{array}{cccc}
			-v & 0 & u-1 & 2 v-1 \\
			0 & u+v & -v-1 & u+2 v-1 \\
			u-1 & -v-1 & -u+2 v+1 & -2 u-v+1 \\
			2 v-1 & u+2 v-1 & -2 u-v+1 & 1 \\
		\end{array}
		\right).$$
		By setting all $3$-minors to zero, we get a zero-dimensional polynomial system in $u,v$. There are $6$ complex solutions and $2$ of them are real matrices of rank $2$. They are
		$$\left(
		\begin{array}{cccc}
			0 & 0 & -1 & -1 \\
			0 & 0 & -1 & -1 \\
			-1 & -1 & 1 & 1 \\
			-1 & -1 & 1 & 1 \\
		\end{array}
		\right)\text{ and }\left(
		\begin{array}{cccc}
			-1 & 0 & 0 & 1 \\
			0 & 2 & -2 & 2 \\
			0 & -2 & 2 & -2 \\
			1 & 2 & -2 & 1 \\
		\end{array}
		\right).$$
		Therefore we can conclude that for all $x=y>0$, the matrix $M$ can be completed to a real symmetric matrix of rank $2$. Repeat this process for all other cells, we can see that the matrix $M$ can be completed to a real symmetric matrix of rank $2$ if and only if $x\ne 0$. This solves the first part of the question.
		
		As for the second part of the question, we recall the Rank-Signature lemma (\cite[Lemma 5.3]{chen2023geometric} or equivalently \cite[Proposition 9.14]{basu2008algorithms}), which says for a continuous family of real symmetric matrices, the signature does not change if the rank remains the same. Using the unique path lifting property of covering maps, we can see that for every pair of points $p,q$ in the same cell $C_k$, the fibers of $p,q$ in $X(\mathbb{R})\backslash Z(\mathbb{R})$ can be connected using paths in $X(\mathbb{R})\backslash Z(\mathbb{R})$. So the endpoints share the same signature and there is a p.s.d.\ in the fiber of some $p_0\in C_k$ if and only if there is a p.s.d.\ in fibers of all $p\in C_k$. Therefore, again it suffices to check only one point in $C_k$. For example, since all the real fibers of $(x,y)=(1,1)$ in $X(\mathbb{R})\backslash Z(\mathbb{R})$ are indefinite, $M$ cannot be completed to a p.s.d.\ for all $x=y>0$. Continuing this procedure, we see that it is impossible to complete $M$ to be a p.s.d.\ of rank $2$ for any $(x,y)$. 
		
	\end{example}
	
	\section{Future Work}	\label{sect:future}
	
	\subsection{The Challenge of Positive-Dimensional Fibers}
	
	In this paper, we focused on the case where fibers are zero-dimensional and we showed that flatness and local constancy of geometric fibers together give a good criterion for being a covering. Meanwhile, families of curves, surfaces, threefolds, etc also play an important role in many aspects of applied mathematics. It would be nice to identify some good families from them.
	
	\begin{problem}\label{prob:positive-dim}
		Let $R$ be a real closed field. The map $\pi :X\to Y$ is an $R$-morphism of $R$-varieties. Find conditions on $\pi$ so that certain topological properties (e.g. non-emptiness, dimension, etc) are shared among the real fibers in a Euclidean neighborhood of each point $y\in Y(R)$.	
	\end{problem}
	
	In particular, one could ask whether a similar statement holds if the ``geometric fiber cardinality'' is replaced by ``Euler characteristic'' or ``cohomology groups'' in the positive-dimensional case. This is far more challenging, as higher-dimensional varieties are very complicated. Even the case for families of curves is highly non-trivial. Still, any progress in this direction would be appealing.
	
	\subsection{Geometry of Discriminant}
	For a finite flat morphism $\pi$, the discriminant locus $D$ is a locally principal subscheme of the base that characterizes the ramification and branching. On the discriminant locus, the geometric fiber cardinality $n(y)$ is lower than the expected number. In Example \ref{example:alg-stats}, we have seen that $n(y)$ is even lower on the singularities. This is not just a coincidence. Dimca and Rosian showed that if one stratifies the discriminant locus of a general monic polynomial $x^n+\sum_{k=0}^{n-1}a_k x^k$ by the geometric fiber cardinality, then this stratification is the canonical Whitney stratification of the discriminant locus \cite{dimca1984samuel}. Similarly, Lazard and McCallum observed that the Hilbert-Samuel multiplicity of discriminant is linked to the singularity type of the fiber \cite{lazard2009iterated}. It would be interesting to establish certain connections between intrinsic geometry of the discriminant $D$ and the number $n(y)$ for $\pi$.
	
	\subsection{Flatness Test}
	
	We have seen that the fiber length can jump for a non-flat morphism in Example \ref{example:alg-stats}. Therefore the importance of flatness is unquestionable. However, the question of whether an algorithm exists to test the flatness of a finite-type general algebra has remained unresolved for decades. Some known cases where such an algorithm exists were reported in \cite{vasconcelos1997flatness}. Similarly, the same question can be asked for the smoothness of a map.
	
	\section*{Acknowledgment}
	
	The author is grateful to Hoon Hong for fruitful discussions on an earlier draft, and to Teresa Krick for explaining the idea behind the construction of a generalized shape lemma, which inspired the use of Chow form recovery in an earlier version of this paper. Above all, the author wishes to express his deepest gratitude to his PhD advisor, Bican Xia, for his continuous support and guidance.

	\printbibliography
\end{document}